\documentclass{amsart}

\usepackage{geometry}
\usepackage[english]{babel}
\usepackage{latexsym, mathrsfs}
\usepackage{amsfonts, amsthm, amsmath, amssymb}
\usepackage{graphicx}
\usepackage[table]{xcolor}
\usepackage{amsaddr}

\newtheorem{prop}{Proposition}[section]
\newtheorem{lem}[prop]{Lemma}
\newtheorem{cor}[prop]{Corollary}
\newtheorem{thm}[prop]{Theorem}
\theoremstyle{definition}

\newtheorem{defi}[prop]{Definition}
\newtheorem{ex}[prop]{Example}

\def\Z{\mathbb{Z}}
\def\N{\mathbb{N}}

\def\equad{\quad \textrm{ and } \quad}
\def\omb{\cellcolor{black!5}}

\def\H{\mathrm{H}}
\def\SMA{\mathrm{SMA}}
\def\MR{\mathrm{MR}}
\def\supp{\mathsf{supp}}

\def\lcm{\mathrm{lcm}}

\def\B{\mathcal{B}}
\def\E{\mathcal{E}}

\def\G{\mathcal{G}}

\def\wt{\widetilde}

\numberwithin{equation}{section}
\newcommand\con{\mathop{+\mkern-10mu+}}

\begin{document}

\title[Magic rectangles, signed magic arrays and\ldots]{Magic rectangles, signed magic arrays 
and \\ integer $\lambda$-fold relative Heffter arrays}

\author[Fiorenza Morini]{Fiorenza Morini}
\address{Dipartimento di Scienze Matematiche, Fisiche e Informatiche, Universit\`a di Parma,\\
Parco Area delle Scienze 53/A, 43124 Parma, Italy}
\email{fiorenza.morini@unipr.it}

\author[Marco Antonio Pellegrini]{Marco Antonio Pellegrini}
\address{Dipartimento di Matematica e Fisica, Universit\`a Cattolica del Sacro Cuore, \\
Via Musei 41, 25121 Brescia, Italy}
\email{marcoantonio.pellegrini@unicatt.it}

\begin{abstract}
Let $m,n,s,k$ be  integers such that $4\leq s\leq n$, $4\leq k \leq m$ and $ms=nk$.
Let $\lambda$ be a divisor of $2ms$ and let $t$ be a divisor of  $\frac{2ms}{\lambda}$.
In this paper we construct magic rectangles $\MR(m,n;s,k)$, signed magic arrays $\SMA(m,n;s,k)$ and 
integer $\lambda$-fold relative Heffter arrays
${}^\lambda \H_t(m,n;s,k)$ where $s,k$ are even integers.
In particular, we prove that there exists an $\SMA(m,n;s,k)$ 
for all $m,n,s,k$ satisfying the previous hypotheses.
Furthermore, we prove that there exist an $\MR(m,n;s,k)$ and an integer ${}^\lambda\H_t(m,n;s,k)$ in each of the following cases:
$(i)$ $s,k \equiv 0 \pmod 4$;
$(ii)$ $s\equiv 2\pmod 4$ and $k\equiv 0 \pmod 4$;
$(iii)$ $s\equiv 0\pmod 4$ and $k\equiv 2 \pmod 4$;
$(iv)$ $s,k\equiv 2 \pmod 4$ and $m,n$ both even. 
\end{abstract}


\maketitle

\section{Introduction}\label{sec:Intro}

In this paper we study partially filled (p.f., for short) arrays, with entries in $\Z$ and whose rows and columns 
have prescribed sums.
In particular, we construct \emph{magic rectangles}, \emph{signed magic arrays} and \emph{integer $\lambda$-fold relative Heffter arrays}.

\begin{defi}\label{def:Magic}
A \emph{signed magic array} $\SMA(m,n; s,k)$ is an $m\times n$ p.f.  array
with elements in $\Omega\subset \Z$, where $\Omega=\{0,\pm 1, \pm 2,\ldots, \pm (ms-1)/2\}$ if $ms$ is odd and
$\Omega=\{\pm 1, \pm 2, \ldots, \pm ms/2\}$ if $ms$ is even, such that
\begin{itemize}
\item[($\rm{a})$] each row contains $s$ filled cells and each column contains $k$ filled cells;
\item[($\rm{b})$] every $x \in \Omega$ appears exactly once in the array;
\item[($\rm{c})$] the elements in every row and column sum to $0$.
\end{itemize}
\end{defi}

The existence of an $\SMA(m,n;s,k)$ has been settle out in the square case (i.e., when $m=n$ and so $s=k$)
and in the tight case (i.e., when $k=m$ and $s=n$), by
Khodkar, Schulz and Wagner.

\begin{thm}\cite{magicDM}\label{square}
There exists an $\SMA(n,n;k,k)$ if and only if either $n=k=1$ or $3\leq k\leq n$.
\end{thm}

\begin{thm}\cite{magicDM}
There exists an $\SMA(m,n;n,m)$  if and only if one of the following cases occurs:
\begin{itemize}
\item[(1)] $m=n=1$;
\item[(2)] $m=2$ and $n\equiv 0,3 \pmod 4$;
\item[(3)] $n=2$ and $m\equiv 0,3\pmod 4$;
\item[(4)] $m,n>2$.
\end{itemize}
\end{thm}

Also the cases when each column contains $2$ or $3$ filled cells have been solved.

\begin{thm}\cite{magic2}
There exists an $\SMA(m,n;s,2)$  if and only if one of the following cases occurs:
\begin{itemize}
\item[(1)] $m=2$ and $n=s\equiv 0,3 \pmod 4$;
\item[(2)] $m,s>2$ and $ms=2n$.
\end{itemize}
\end{thm}

\begin{thm}\cite{magic3}
There exists an $\SMA(m,n;s,3)$  if and only if $3\leq m,s\leq n$ and $ms=3n$.
\end{thm}

In this paper we consider the case when $s$ and $k$ are both even, proving the following result.

\begin{thm}\label{main}
Let $s,k$ be two even integers with $s,k\geq 4$.
There exists an $\SMA(m,n;s,k)$ if and only if $4\leq s\leq n$, $4\leq k \leq m$ and $ms=nk$.
\end{thm}

This result will be obtained working in the more general context of the integer $\lambda$-fold relative Heffter arrays.

In \cite{A} Dan Archdeacon introduced an important class of p.f.  arrays, called
\emph{Heffter arrays}. One of the applications of these objects is that they allow, under suitable conditions,
to construct pairs of cyclic cycle decompositions of the complete graph $K_v$ on $v$ vertices.
With the aim to extend this application to complete multipartite graphs, in \cite{RelH} the authors of the present paper,
in collaboration with Costa and Pasotti, proposed a first generalization of Archdeacon's idea introducing p.f. arrays called
\emph{relative Heffter arrays}. A further generalization, that allows to work with complete multipartite multigraphs,
was introduced in \cite{SA} by Costa and Pasotti. These new objects are called \emph{$\lambda$-fold relative Heffter arrays}.
We recall here their definition, where we denote by $\E(A)$ the \emph{list} of the entries of the filled cells of a
p.f. array $A$.

\begin{defi}\label{LambdaH}
Let $m,n,s,k,t,\lambda$ be positive integers such that $\lambda$ divides $2ms$ and $t$ divides $\frac{2ms}{\lambda}$.
Let $J$ be the subgroup of order $t$ of $\Z_v$, where $v=\frac{2ms}{\lambda}+t$.
A \emph{$\lambda$-fold Heffter array} over $\Z_v$ relative to $J$, denoted by ${}^\lambda \H_t(m,n;s,k)$, is
an $m\times n$ p.f. array $A$ with elements in $\Omega=\Z_v\setminus J$ such that:
\begin{itemize}
\item[{\rm (a)}] each row contains $s$ filled cells and each column contains $k$ filled cells;
\item[{\rm (b)}] every element of $\Omega$ appears exactly $\lambda$ times in the list $\E(A)\cup -\E(A)$;
\item[{\rm (c)}] the elements in every row and column sum to $0$.
\end{itemize}
\end{defi}

Item (b) of the previous definition requires some explications.
The additive group $\Z_v$ contains an involution if and only if $v$ is even: in this case,  the unique involution
$\iota\in \Z_v$ belongs to $\Omega$ if and only if $t$ is odd.
We observe that the assumption $v$ even and  $t$ odd implies that $\lambda$ is even and does not divide $ms$.
So, we can write (b) as follows:
if $\Omega$ does not contain involutions, every $x \in \Omega$ appears in $A$, up to sign, exactly $\lambda$ times;
if $\Omega$ contains the  involution $\iota$, then
every $x \in \Omega\setminus \{\iota\}$ appears, up to sign, exactly $\lambda$ times, while
$\iota$ appears exactly $\frac{\lambda}{2}$ times.

Instead of working in a finite cyclic group, one can construct $\lambda$-fold relative
Heffter arrays whose entries are rational integers. In this case, the previous definition becomes as follows.

\begin{defi}\label{LambdaHInt}
Let $m,n,s,k,t,\lambda$ be positive integers such that $\lambda$ divides $2ms$ and $t$ divides $\frac{2ms}{\lambda}$.
Let 
$$\Phi=\left\{1,2,\ldots,\left\lfloor\frac{v}{2} \right\rfloor   \right\}\setminus \left\{\ell, 2\ell,\ldots,
\left\lfloor \frac{t}{2}\right\rfloor \ell \right \} \subset \Z,\quad \textrm{where } 
v=\frac{2ms}{\lambda}+t\equad \ell=\frac{v}{t}.$$ 
An \emph{integer} ${}^\lambda \H_t(m,n;s,k)$ is
an $m\times n$ p.f. array with elements in $\Phi$ such that:
\begin{itemize}
\item[{\rm(a)}] each row contains $s$ filled cells and each column contains $k$ filled cells;
\item[{\rm(b)}] if  $v$ is odd or if  $t$ is even,  every element of $\Phi$ appears, up to sign, exactly 
$\lambda$ times in the array; if $v$ is even and  $t$ is odd, every element of $\Phi\setminus \{\frac{v}{2}\}$
appears, up to sign, exactly 
$\lambda$ times while $\frac{v}{2}$ appears, up to sign, exactly $\frac{\lambda}{2}$ times;
\item[{\rm(c)}] the elements in every row and column sum to $0$.
\end{itemize}
\end{defi}

Observe that when $\lambda=1$ one retrieves the concept of (integer) relative Heffter array. In particular, an
(integer) ${}^1\H_1(m,n;s,k)$ is exactly a classical (integer) Heffter array, as defined by Archdeacon.
The problem of the existence of \emph{square} classical Heffter arrays has been completely solved
in \cite{ADDY,DW} for the integer case, and in \cite{CDDY} for the general case. For the other cases (non-square or relative),
partial results have been obtained in \cite{ABD,RelHBiem,MP}.
Applications of (relative) Heffter arrays to graph decompositions and biembeddings are described, for instance, in
\cite{BCDY,CDYBiem,CMPPHeffter,DM}.

Here, we prove the following result.

\begin{thm}\label{mainH}
Let $m,n,s,k$ be  integers such that $4\leq s\leq n$, $4\leq k \leq m$ and $ms=nk$.
Let $\lambda$ be a divisor of $2ms$ and let $t$ be a divisor of  $\frac{2ms}{\lambda}$.
There exists an integer ${}^\lambda\H_t(m,n;s,k)$ in each of the following cases:
\begin{itemize}
\item[$(1)$] $s,k \equiv 0 \pmod 4$;
\item[$(2)$] $s\equiv 2\pmod 4$ and $k\equiv 0 \pmod 4$;
\item[$(3)$] $s\equiv 0\pmod 4$ and $k\equiv 2 \pmod 4$;
\item[$(4)$] $s,k\equiv 2 \pmod 4$ and  $m,n$  both even. 
\end{itemize}
\end{thm}

Looking at Definitions \ref{def:Magic} and \ref{LambdaHInt} the reader can easily see that, when $ms$ is even,
 a signed magic array is a particular integer $2$-fold relative Heffter array.
In fact, the integer ${}^2\H_1(m,n;s,k)$ we construct in the following sections is actually a signed magic array
$\SMA(m,n;s,k)$.
So, Theorem \ref{main} will follow from Theorem \ref{mainH}, except when $s,k\equiv 2\pmod 4$ and $m,n$ are odd.
Nevertheless, for these exceptional values, we will construct $\SMA(m,n;s,k)$ starting
from \emph{square} signed magic arrays, whose existence is assured by Theorem \ref{square},
and exploiting the flexibility of our constructions.

Our results on signed magic arrays allow us to build also magic rectangles.

\begin{defi}\label{def:rect}
A \emph{magic rectangle} $\MR(m,n; s,k)$ is an $m\times n$ p.f.  array
with elements in $\Omega=\{0,1,\ldots,ms-1\}\subset \Z$ such that
\begin{itemize}
\item[($\rm{a})$] each row contains $s$ filled cells and each column contains $k$ filled cells;
\item[($\rm{b})$] every $x \in \Omega$ appears exactly once in the array;
\item[($\rm{c})$] the sum of the elements in each row is a constant value $c_1$  and
the sum of the elements in each column is a constant value $c_2$.
\end{itemize}
\end{defi}
Clearly, in the previous definition we must have $c_1=\frac{s(ms-1)}{2}$ and $c_2=\frac{k(ms-1)}{2}$.
The reader can find results on the existence of these objects in \cite{rect3,rect1} and in the references within.
Here, we prove the following.

\begin{thm}\label{mainR}
Let $m,n,s,k$ be integers such that $4\leq s\leq n$, $4\leq k \leq m$ and $ms=nk$.
There exists an $\MR(m,n;s,k)$ in each of the following cases:
\begin{itemize}
\item[$(1)$] $s,k \equiv 0 \pmod 4$;
\item[$(2)$] $s\equiv 2\pmod 4$ and $k\equiv 0 \pmod 4$;
\item[$(3)$] $s\equiv 0\pmod 4$ and $k\equiv 2 \pmod 4$;
\item[$(4)$] $s,k\equiv 2 \pmod 4$ and  $m,n$  both even. 
\end{itemize}
\end{thm}

\section{Notations}

In this paper,  the arithmetic on the row (respectively, on the column) indices is performed modulo $m$
(respectively, modulo $n$), where the set of reduced residues is $\{1,2,\ldots,m\}$ (respectively,
 $\{1,2,\ldots,n\}$), while the entries of the arrays are taken in $\Z$.
Given two integers $a\leq b$, we denote by $[a,b]$ the interval consisting of the integers $a,a+1,\ldots,b$.
If $a>b$, then $[a,b]$ is empty. We denote by $(i,j)$ the cell in the $i$-th row and $j$-th column
of an array $A$. The \emph{support} of $A$, denoted by $\supp(A)$,
is defined to be the set of the absolute values of the elements contained in $A$.

If $A$ is an $m\times n$ p.f. array, for $i\in[1,n]$ we define the $i$-th diagonal as
$$D_i=\{(1,i),(2,i+1),\ldots,(m,i+m-1)\}.$$

\begin{defi}
A p.f. array with entries in $\Z$ is said to be \emph{shiftable} if
every row and every column contains an equal number of positive and negative entries.
\end{defi}

Let $A$ be a shiftable p.f. array and $x$ be a nonnegative integer.
Let $A\pm x$ be the (shiftable) p.f. array obtained adding $x$ to each positive
entry of $A$ and $-x$   to each negative entry of $A$.
Observe that, since  $A$ is shiftable, the row and column sums of $A\pm x$ are exactly
the row and column sums of $A$.

We denote by $\tau_i(A)$ and $\gamma_j(A)$ the sum of the elements of the $i$-th row and the sum of the
elements of the $j$-th column, respectively, of  a p.f. array $A$.

For a block $B$, we write $\mu(B)=\mu$ if every element of $\supp(B)$
appears, up to sign, exactly $\mu$ times in $\E(B)$.

Given a sequence $S=(B_1,B_2,\ldots,B_r)$ of shiftable p.f. arrays and a nonnegative integer $x$,
we write $S\pm x$ for the sequence $(B_1\pm x, B_2\pm x,\ldots,B_r\pm x)$.
We set $\E(S) = \cup_i \E(B_i)$ and $\supp(S)=\cup_i \supp(B_i)$.
We also write $\mu(S)=\mu$ if $\mu(B_i)=\mu$ for all $i$.

If $S_1=(a_1,a_2,\ldots,a_{r})$ and $S_2=(b_1,b_2,\ldots,b_{u})$ are two sequences,
by $S_1 \con S_2$ we mean
the sequence $(a_1,a_2,\ldots,a_r, b_1, b_2, \ldots, b_u)$ obtained by concatenation
of $S_1$ and $S_2$. In particular, if $S_1$ is the empty sequence then $S_1 \con S_2=S_2$.
Furthermore, given the sequences $S_1,\ldots,S_c$, we write
$\con\limits_{i=1}^c S_i$  for $(\cdots((S_1 \con S_2) \con S_3) \con  \cdots) \con  S_c$.

Given a positive integer $n$ and a sequence $S=(a_1,a_2,\ldots,a_{r})$, we denote
by $n\ast S$ the sequence obtained concatenating $n$ copies of $S$.

Finally, we recall that the support of an integer ${}^\lambda \H_t(m,n;s,k)$ is the set
$$\Phi=\left[1, \left\lfloor\frac{t\ell}{2} \right\rfloor   \right]\setminus \left\{\ell, 2\ell,\ldots,
\left\lfloor \frac{t}{2}\right\rfloor \ell \right \},\quad \textrm{where } \ell=\frac{2ms}{\lambda t}+1=\frac{v}{t}.$$
Note that, if $\lambda$ divides $ms$, then 
$$\Phi=\left[1, \frac{ms}{\lambda}+\left\lfloor\frac{t}{2} \right\rfloor   \right]\setminus \left\{\ell, 2\ell,\ldots,
\left\lfloor \frac{t}{2}\right\rfloor \ell \right \}.$$
Also, every element of $\Phi$ appears in ${}^\lambda \H_t(m,n;s,k)$, up to sign, exactly $\lambda$ times.
If $\lambda$ does not divide $ms$, in order to obtain an integer ${}^\lambda \H_t(m,n;s,k)$, we have to construct a p.f. array $A$ such that
\begin{equation}\label{cond2ms}
\begin{array}{l}
\textrm{if } \ell \textrm{ is odd or if } t \textrm{ is even, every element of }
\Phi \textrm{ appears in } A,  \textrm{up to sign, exactly } \lambda  \\
\textrm{times; otherwise, i.e, if } \ell \textrm{ is even and } t  
\textrm{ is odd, every element of } \Phi \setminus \left\{\frac{t\ell}{2}\right\}\textrm{ appears} \\
\textrm{in } A, 
\textrm{up to sign, exactly } \lambda \textrm{ times, while the integer }
\frac{t\ell}{2} \textrm{ appears, up to sign, } \frac{\lambda}{2} \textrm{ times.}
\end{array}
\end{equation}

\section{The case $s,k\equiv 0 \pmod 4$}\label{s0k0}

In this section we prove the existence of an integer
${}^\lambda\H_t(m,n;s,k)$ when both $s$ and $k$ are divisible by $4$.
First of all, we set
$$d=\gcd(m,n),\quad m=d\bar m,\quad n= d \bar n,\quad s=4 \bar s \equad k=4\bar k.$$
Note that from $ms=nk$ we obtain that $\bar n$ divides $\bar s$ and $\bar m$ divides $\bar k$. Hence, we can write $\bar
s= c \bar n$
and $\bar k=c \bar m$.

Fix two integers $a,b\geq 0$ and consider the following shiftable p.f. array:
$$B=B_{a,b}=\begin{array}{|c|c|}\hline
     1 & -(a+1) \\\hline
      & \\\hline
   -(b+1) & a+b+1\\\hline
    \end{array}\;.$$
Note that the sequences of the  row/column sums are $(-a,a)$ and $(-b,b)$, respectively.
We  use this $3\times 2$ block for constructing  p.f. arrays whose rows and columns sum to zero.
Start taking an empty $m\times n$ array $A$, fix
$m\bar n$ nonnegative integers $y_0, y_1,\ldots,y_{m\bar n-1}$, and arrange the blocks
$B\pm y_j$ in such a way that the element $1+y_j$ fills the cell
$(j+1,j+1)$ of $A$ (recall that we work modulo $m$ on row indices  and modulo $n$ on column indices).
In this way, we fill the diagonals $D_{im-1}, D_{im}, D_{im+1}, D_{im+2}$ with $i\in [1,\bar n]$.
In particular, every row has $4\bar n$ filled cells and every column has $4\bar m$ filled cells.

Looking at the rows, the elements belonging to the
diagonals $D_{im+1},D_{im+2}$ sum to $-a$, while the elements belonging to the diagonals $D_{im-1},D_{im}$ sum to $a$.
Looking at the columns, the elements belonging to the diagonals $D_{im+1},D_{im-1}$
sum to $-b$, while the elements belonging to the diagonals $D_{im+2},D_{im}$
sum to $b$. Then $A$ has row/column sums equal to zero.

Applying this process $c$ times (working with the diagonals $D_{im+3},D_{im+4},$ $D_{im+5},D_{im+6}$, and so on),
we obtain a p.f. array $A$, whose rows have exactly $4\bar n \cdot c=s$ filled cells
and whose columns have exactly $4\bar m \cdot c =k$ filled cells.

\begin{ex}
For $a=2$ and $b=5$, fixing the integers $0,1, 10, 11, 20, 21, 30, 31, 40, 41, 50, 51,$ we can fill the diagonals
$D_1,D_2,D_5, D_6, D_7, D_8, D_{11}, D_{12}$ of the following $6\times 12$ p.f. array,
where we highlighted the block $B_{2,5}$:

\begin{footnotesize}
$$A=\begin{array}{|c|c|c|c|c|c|c|c|c|c|c|c|}\hline
\omb 1  & \omb -3 &      &       &  -26  &   28 &  31  &  -33 &     &      &  -56 & 58   \\\hline
      59 &        2 &   -4 &       &       &  -27 &   29 &  32  & -34 &      &      & -57    \\\hline
\omb -6 &  \omb 8 &   11 &   -13 &       &      & -36  &  38  & 41  &  -43 &      &     \\\hline
         &       -7 &   9  &    12 &  -14  &      &      & -37  & 39  &  42  & -44  &     \\\hline
         &          &  -16 &   18  &    21 &  -23 &      &      & -46 &  48  &  51  & -53 \\\hline
 -54     &          &      &   -17 &    19 &   22 &  -24 &      &     &  -47 &  49  &  52    \\\hline
  \end{array}\;.$$
\end{footnotesize}

\noindent
Note that $\supp(A)=[1,60]\setminus\{5j: j \in [1, 12]\}$.
As the reader can verify, $A$ is an integer ${}^1\H_{24}(6,12;8,4)$:
in this case $\ell=\frac{2\cdot 6\cdot 8}{24}+1=5$.
\end{ex}

The constructions we present in this section are obtained following this procedure,
so they all produce shiftable p.f. arrays of size $m\times n$ whose rows and columns sum to zero.

Here we always assume that
$4\leq s \leq n$, $4\leq k \leq m$, $ms=nk$ and $s,k\equiv 0 \pmod 4$.
Let $\lambda$ be a divisor  of $2ms$ and $t$ be a divisor of $\frac{2ms}{\lambda}$; set
$$\ell=\frac{2ms}{\lambda t}+1.$$

We first consider the case when $\lambda$ divides $ms$.
To obtain an integer ${}^\lambda\H_t(m,n;s,k)$ with $s,k\equiv 0 \pmod 4$, we only have to determine two integers 
$a,b\geq 0$
and a set $X=\left\{x_0,x_1,\ldots,\right.$ $\left. x_{f-1}\right\}\subset \N$ such that $\mu(B_{a,b})=\mu$
divides $\lambda$  and $\bigcup\limits_{x \in X} \supp(B_{a,b}\pm x)=\Phi$.
Note that $f=\frac{ms}{4}\frac{\mu}{\lambda}$. So we can take the sequence $Y= \frac{\lambda}{\mu}\ast \left(x_0,x_1,\ldots,x_{f-1}\right)$. Writing
$Y=(y_0,y_1,\ldots, y_{\frac{ms}{4}-1})$ we construct $A$ using the blocks $B_{a,b}\pm y_j$. In this way, every element of $\supp(A)$ occurs, up the sign, $\lambda$ times in $A$.
For instance, we can arrange the blocks in such a way that the element $1+y_j$ fills
the cell $(j+1,4q_j+j+1)$, where $q_j$ is the quotient of the division of $j$ by $\lcm(m,n)$.

\begin{lem}\label{lambda4}
Let $\lambda$ be a divisor of $ms$ such that
$\lambda\equiv 0 \pmod 4$. There  exists an integer ${}^\lambda\H_t(m,n;s,k)$ for any divisor $t$ of $\frac{2ms}{\lambda}$.
\end{lem}

\begin{proof}
Let  $B=B_{0,0}=\begin{array}{|c|c|}\hline
     1 & -1 \\\hline
      & \\\hline
   -1 & 1\\\hline
    \end{array}\;$. Note that $\mu(B)=4$.
An integer ${}^\lambda\H_t(m,n;s,k)$, say $A$, can be obtained following the construction described before, once we exhibit a
suitable set $X$
of size $\frac{ms}{\lambda}$, in such a way that $\supp(A)=\Phi$.
Consider the set $X=\{i-1\mid i\in \Phi \}$ of size $\frac{ms}{\lambda}$: clearly,
$\bigcup\limits_{x \in X} \supp(B\pm x)=\Phi$.
Now we take $\frac{\lambda}{4}$ copies of every block $B\pm x$:
the p.f. array $A$ obtained following our procedure is an integer ${}^\lambda\H_t(m,n;s,k)$.
\end{proof}

For instance, to construct an integer ${}^8\H_{5}(5,10;8,4)$ we can follow the proof of the previous lemma.
In fact, $\lambda=8$ and $t=5$ divides $\frac{2\cdot 5\cdot 8}{8}$; note that $\ell=3$ and $Y=2\ast (0,1,3,4,6)$.

\begin{footnotesize}
$${}^8\H_{5}(5,10;8,4)=\begin{array}{|c|c|c|c|c|c|c|c|c|c|}\hline
\omb    1&  \omb  -1&     &   -5&    5&    1&   -1&     &   -5&    5  \\\hline
   7&     2&    -2&     &   -7&    7&    2&   -2&     &   -7 \\\hline
\omb  -1&  \omb  1&     4&    -4&     &   -1&    1&    4&   -4&     \\\hline
    &   -2&    2&     5&   -5&     &   -2&    2&    5&   -5 \\\hline
  -7&     &   -4&    4&     7&   -7&     &   -4&    4&    7 \\\hline
  \end{array}\;.$$
\end{footnotesize}

\begin{lem}\label{lambda2}
Let $\lambda$ be a divisor of $ms$ such that
$\lambda\equiv 2 \pmod 4$. There  exists an integer ${}^\lambda\H_t(m,n;s,k)$ for any divisor $t$ of $\frac{2ms}{\lambda}$.
\end{lem}

\begin{proof}
We first consider the case when $\ell$ is odd, which means that $t$ divides $\frac{ms}{\lambda}$.
Let  $B=B_{1,0}=\begin{array}{|c|c|}\hline
     1 & -2 \\\hline
      & \\\hline
   -1 & 2\\\hline
    \end{array}\;$; note that $\mu(B)=2$.
We start considering the set $X_0=\{0,2,4,\ldots,\ell-3 \}$ of size
$\frac{\ell-1}{2}=\frac{ms}{\lambda t}$: it is easy to see that
$\bigcup\limits_{x \in X_0} \supp(B\pm x)=[1,\ell]\setminus
\{\ell\}$.
Similarly, for any $i \in \N$,  if $X_i=\{i\ell,i\ell+2,i\ell+4,\ldots, (i+1)\ell-3\}$, then
$$\bigcup_{x \in X_i} \supp(B\pm x)=[i\ell+1, (i+1)\ell ]\setminus
\{(i+1)\ell \}$$
and $X_{i_1}\cap X_{i_2}=\emptyset$ if $i_1\neq i_2$.

If $t$ is even, take $X=\bigcup\limits_{i=0}^{t/2-1} X_i$: this is a set of size
$\frac{t}{2}\cdot \frac{ms}{\lambda t}=\frac{ms}{2\lambda}$, as required.
Furthermore,
$$\begin{array}{rcl}
\bigcup\limits_{x \in X} \supp(B\pm x) & =& \bigcup\limits_{i=0}^{t/2-1}
\left( [i\ell+1, (i+1)\ell ]\setminus
\{(i+1)\ell \} \right)\\[8pt]
 & = & \left[1,\frac{t}{2}\ell\right]\setminus \left\{\ell, 2\ell, \ldots, \frac{t}{2}\ell\right\}
=\left[1,\frac{ms}{\lambda}+\frac{t}{2}\right]\setminus  \left\{\ell, 2\ell, \ldots, \frac{t}{2}\ell\right\}.
\end{array}$$

Suppose now that $t$ is odd, which implies that $\ell\equiv 1 \pmod 4$.  Take
$$Z=\left\{\left(\frac{t-1}{2}\right)\ell, \left(\frac{t-1}{2}\right)\ell+2, \left(\frac{t-1}{2}\right)\ell+4,\ldots,
\left(\frac{t-1}{2}\right)\ell+2 \frac{\ell-5}{4} \right\}.$$
Then  $|Z|=\frac{\ell-1}{4}=\frac{ms}{2\lambda t}$ and
$\bigcup\limits_{z \in Z} \supp(B\pm z)=\left[\left(\frac{t-1}{2}\right)\ell+1,
\left(\frac{t-1}{2}\right)\ell+\frac{\ell-1}{2} \right]$.
So, we can take $X=\left(\bigcup\limits_{i=0}^{(t-3)/2} X_i\right)\cup Z$: this is a set of size
$\frac{t-1}{2}\cdot \frac{ms}{\lambda t}+\frac{ms}{2\lambda t}=\frac{ms}{2\lambda}$,
as required.
In this case, $$\begin{array}{rcl}
\bigcup\limits_{x \in X} \supp(B\pm x) & =& \bigcup\limits_{i=0}^{\frac{t-3}{2}}
\left( [i\ell+1, (i+1)\ell ]\setminus
\{(i+1)\ell \} \right)\cup \left[\left(\frac{t-1}{2}\right)\ell+1,
\left(\frac{t-1}{2}\right)\ell+\frac{\ell-1}{2}
\right]\\[8pt]
 & = & \left(\left[1, \frac{t-1}{2}\ell \right]\setminus \left\{\ell, 2\ell, \ldots,
\frac{t-1}{2}\ell\right\}\right)\cup
\left[\left(\frac{t-1}{2}\right)\ell+1, \frac{ms}{\lambda}+\frac{t-1}{2} \right]
\\[8pt]
& =& \left[1,\frac{ms}{\lambda}+\left\lfloor\frac{t}{2}\right\rfloor\right]\setminus
\left\{\ell, 2\ell, \ldots, \left\lfloor\frac{t}{2}\right\rfloor\ell\right\}.
\end{array}$$

In both cases, considering $\frac{\lambda}{2}$ copies of the distinct blocks $B\pm x$ with $x\in X$, the p.f. array $A$ obtained
following our procedure is an integer ${}^\lambda\H_t(m,n;s,k)$.

Finally, we consider the case when $\ell$ is even, which implies that $t\equiv 0 \pmod 4$.
Let  $B=B_{\ell,0}=\begin{array}{|c|c|}\hline
     1 & -(\ell+1) \\\hline
      & \\\hline
   -1 & \ell+1\\\hline
    \end{array}\;$; note that $\mu(B)=2$.
We start considering the set $X_0=[0,\ell-2]$ of size
$\ell-1=\frac{2ms}{\lambda t}$: it is easy to see that
$\bigcup\limits_{x \in X_0} \supp(B\pm x)=[1,2\ell]\setminus
\{\ell,2\ell\}$.
Similarly, for any $i \in \N$,  if $X_i=[2i\ell,(2i+1)\ell-2]$, then
$$\bigcup_{x \in X_i} \supp(B\pm x)=[2i\ell+1, (2i+2)\ell ]\setminus
\{ (2i+1)\ell,(2i+2)\ell \}$$
and $X_{i_1}\cap X_{i_2}=\emptyset$ if $i_1\neq i_2$.
Take $X=\bigcup\limits_{i=0}^{t/4-1} X_i$: this is a set of size
$\frac{t}{4}\cdot (\ell-1)=\frac{ms}{2\lambda}$, as required.
In this case,
$$\begin{array}{rcl}
\bigcup\limits_{x \in X} \supp(B\pm x) & =& \bigcup\limits_{i=0}^{t/4-1}
\left( [2i\ell+1, (2i+2)\ell ]\setminus
\{(2i+1)\ell,(2i+2)\ell \} \right)\\[8pt]
 & = & \left[1,\frac{t}{2}\ell\right]\setminus \left\{\ell, 2\ell, \ldots, \frac{t}{2}\ell\right\}
=\left[1,\frac{ms}{\lambda}+\frac{t}{2}\right]\setminus  \left\{\ell, 2\ell, \ldots, \frac{t}{2}\ell\right\}.
\end{array}$$
Now we take $\frac{\lambda}{2}$ copies of every block $B\pm x$: the p.f. array $A$ obtained following our procedure is an 
integer ${}^\lambda\H_t(m,n;s,k)$.
\end{proof}

We now deal with the case $\lambda$ odd. This implies that $\lambda$ divides $\frac{ms}{4}$.

\begin{lem}\label{k4-3}
Let $\lambda$ be a positive odd integer. There  exists an integer ${}^\lambda\H_t(m,n;s,k)$ for any divisor $t$ of $\frac{2ms}{\lambda}$ such that $t\equiv 0 \pmod 8$.
\end{lem}

\begin{proof}
Let  $B=B_{\ell,2\ell}=\begin{array}{|c|c|}\hline
     1 & -(\ell+1) \\\hline
      & \\\hline
   -(2\ell+1) & 3\ell +1\\\hline
    \end{array}\;$, where $\ell=\frac{2ms}{\lambda t}+1$. Note that $\mu(B)=1$.
An integer ${}^\lambda\H_t(m,n;s,k)$, say $A$, can be obtained following the construction described before, once we exhibit a
suitable set $X$
of size $\frac{ms}{4\lambda}$, in such a way that $\supp(A)=\left[1,\frac{ms}{\lambda}+\frac{t}{2}\right]\setminus  \left\{\ell, 2\ell,
\ldots, \frac{t}{2}\ell\right\}$.

Start considering the set $X_0=[0,\ell-2]$ of size $\ell-1=\frac{2ms}{\lambda t}$: it is easy to see that
$\bigcup\limits_{x \in X_0} \supp(B\pm x)=[1,4\ell]\setminus \{\ell, 2\ell, 3\ell, 4\ell \}$.
Similarly, for any $i \in \N$,  if $X_i=[4i\ell, (4i+1)\ell-2 ]$, then
$$\bigcup_{x \in X_i} \supp(B\pm x)=[4i\ell+1, (4i+4)\ell ]
\setminus \{(4i+1)\ell, (4i+2)\ell, (4i+3)\ell, (4i+4)\ell \}.$$
Clearly, $X_{i_1}\cap X_{i_2}=\emptyset$ if $i_1\neq i_2$.
So, take $X=\bigcup\limits_{i=0}^{t/8-1} X_i$: this is a set of size $\frac{t}{8}\cdot (\ell-1)
=\frac{t}{8}\cdot \frac{2ms}{\lambda t}=\frac{ms}{4\lambda}$, as
required.
It is easy to see that
$$\begin{array}{rcl}
\bigcup\limits_{x \in X} \supp(B\pm x) & =& \bigcup\limits_{i=0}^{t/8-1 } \left([4i\ell+1, (4i+4)\ell ] \setminus
 \{(4i+1)\ell, (4i+2)\ell, (4i+3)\ell, (4i+4)\ell \}\right)\\[8pt]
 & = & \left[1,\frac{t}{2}\ell\right]\setminus \left\{\ell, 2\ell, \ldots, \frac{t}{2}\ell\right\}
=\left[1,\frac{ms}{\lambda}+\frac{t}{2}\right]\setminus  \left\{\ell, 2\ell, \ldots, \frac{t}{2}\ell\right\}.
\end{array}$$
Now we take $\lambda$ copies of every block $B\pm x$: the p.f. array $A$ obtained following our procedure is an integer 
${}^\lambda\H_t(m,n;s,k)$.
\end{proof}

\begin{lem}\label{k4-2}
Let $\lambda$ be a positive odd integer. There exists an integer ${}^\lambda \H_t(m,n;s,k)$ for any divisor $t$ of
$\frac{ms}{\lambda}$ such that $t\equiv 0 \pmod 4$.
\end{lem}

\begin{proof}
Let  $B=B_{1,\ell}=\begin{array}{|c|c|}\hline
     1 & -2 \\\hline
      & \\\hline
   -(\ell+1) & \ell+2\\\hline
    \end{array}:$ note that $\mu(B)=1$ and, since $t$ divides $\frac{ms}{\lambda}$, $\ell=\frac{2ms}{\lambda t}+1$
    is an odd integer.
We start considering the set $X_0=\{0,2,4,\ldots,\ell-3\}$ of size
$\frac{\ell-1}{2}=\frac{ms}{\lambda t}$: it is easy to see that
$\bigcup\limits_{x \in X_0} \supp(B\pm x)=[1,\ell-1]\cup [\ell+1,2\ell-1]=[1,2\ell]\setminus
\{\ell, 2\ell\}$.
Similarly, for any $i \in \N$,  if $X_i=\{2i\ell,2i\ell+2,2i\ell+4,\ldots, (2i+1)\ell-3 \}$, then
$$\bigcup_{x \in X_i} \supp(B\pm x)=[2i\ell+1, 2(i+1)\ell ]\setminus
\{(2i+1)\ell, (2i+2)\ell \}$$
and $X_{i_1}\cap X_{i_2}=\emptyset$ if $i_1\neq i_2$.
So, take $X=\bigcup\limits_{i=0}^{t/4-1} X_i$: this is a set of size
$\frac{t}{4}\cdot \frac{\ell-1}{2} =\frac{t}{4}\cdot \frac{ms}{\lambda t}=\frac{ms}{4 \lambda}$, as required.
Hence, 
$$\begin{array}{rcl}
\bigcup\limits_{x \in X} \supp(B\pm x) & =& \bigcup\limits_{i=0}^{t/4-1}
\left([2i\ell+1, 2(i+1)\ell ]\setminus
\{(2i+1)\ell, (2i+2)\ell \} \right)\\[8pt]
 & = & \left[1,\frac{t}{2}\ell\right]\setminus \left\{\ell, 2\ell, \ldots, \frac{t}{2}\ell\right\}
=\left[1,\frac{ms}{\lambda}+\frac{t}{2}\right]\setminus  \left\{\ell, 2\ell, \ldots, \frac{t}{2}\ell\right\}.
\end{array}$$
Now we take $\lambda$ copies of every block $B\pm x$:  the p.f. array $A$ obtained following our procedure is an integer
${}^\lambda\H_t(m,n;s,k)$.
\end{proof}

For instance, to construct an integer ${}^5\H_{4}(5,10;8,4)$ we can follow the proof of the previous lemma.
In fact, $\lambda=5$ and $t=4$ divides $\frac{5\cdot 8}{5}$; note that $\ell=5$ and $Y=5\ast (0,2)$.

\begin{footnotesize}
$${}^5\H_{4}(5,10;8,4)=\begin{array}{|c|c|c|c|c|c|c|c|c|c|}\hline
\omb    1&  \omb  -2&     &   -8&    9&    3&   -4&     &   -6&    7  \\\hline
   9&     3&    -4&     &   -6&    7&    1&   -2&     &   -8 \\\hline
\omb  -6&  \omb  7&     1&    -2&     &   -8&    9&    3&   -4&     \\\hline
    &   -8&    9&     3&   -4&     &   -6&    7&    1&   -2 \\\hline
  -4&     &   -6&    7&     1&   -2&     &   -8&    9&    3 \\\hline
  \end{array}\;.$$
\end{footnotesize}

\begin{lem}\label{k4-1}
Let $\lambda$ be a positive odd integer. There exists an integer ${}^\lambda\H_t(m,n;s,k)$ for any divisor $t$ of $\frac{ms}{2\lambda}$.
\end{lem}

\begin{proof}
Let  $B=B_{1,2}=\begin{array}{|c|c|}\hline
     1 & -2 \\\hline
      & \\\hline
   -3 & 4\\\hline
    \end{array}\;$. Note that  $\mu(B)=1$ and  $\ell=\frac{2ms}{\lambda t}+1\equiv 1 \pmod 4$ since $t$ divides $\frac{ms}{2\lambda}$.
We start considering the set $X_0=\{0,4,8,\ldots,\ell-5 \}$ of size
$\frac{\ell-1}{4}=\frac{ms}{2\lambda t}$: clearly,
$\bigcup\limits_{x \in X_0} \supp(B\pm x)=[1,\ell]\setminus
\{\ell\}$.
Similarly, for any $i \in \N$,  if $X_i=\{i\ell,i\ell+4,i\ell+8,\ldots, (i+1)\ell-5\}$, then
$$\bigcup_{x \in X_i} \supp(B\pm x)=[i\ell+1, (i+1)\ell ]\setminus
\{(i+1)\ell \}$$
and $X_{i_1}\cap X_{i_2}=\emptyset$ if $i_1\neq i_2$.

If $t$ is even, take $X=\bigcup\limits_{i=0}^{t/2-1} X_i$: this is a set of size
$\frac{t}{2}\cdot \frac{\ell-1}{4}=\frac{t}{2}\cdot \frac{ms}{2\lambda t}=\frac{ms}{4 \lambda}$, as required.
Hence,
$$\begin{array}{rcl}
\bigcup\limits_{x \in X} \supp(B\pm x)  & =& \bigcup\limits_{i=0}^{t/2-1}
\left( [i\ell+1, (i+1)\ell ]\setminus
\{(i+1)\ell \} \right)\\[8pt]
 & = & \left[1,\frac{t}{2}\ell\right]\setminus \left\{\ell, 2\ell, \ldots, \frac{t}{2}\ell\right\}
=\left[1,\frac{ms}{\lambda}+\frac{t}{2}\right]\setminus  \left\{\ell, 2\ell, \ldots, \frac{t}{2}\ell\right\}.
\end{array}$$
Suppose now that $t$ is odd. Notice that, in this case, $\ell\equiv 1\pmod 8$. Take
$$Z=\left\{\left(\frac{t-1}{2}\right)\ell, \left(\frac{t-1}{2}\right)\ell+4, \left(\frac{t-1}{2}\right)\ell+8,\ldots,
\left(\frac{t-1}{2}\right)\ell+ 4\frac{\ell-9}{8} \right\}.$$
Then  $|Z|=\frac{\ell-1}{8}=\frac{ms}{4\lambda t}$ and
$\bigcup\limits_{z \in Z} \supp(B\pm z)=\left[\left(\frac{t-1}{2}\right)\ell+1,
\left(\frac{t-1}{2}\right)\ell+\frac{\ell-1}{2} \right]$.
Take $X=\left(\bigcup\limits_{i=0}^{(t-3)/2} X_i\right)\cup Z$: this is a set of size
$\frac{t-1}{2}\cdot\frac{\ell-1}{4}+\frac{\ell-1}{8} =\frac{t-1}{2}\cdot \frac{ms}{2\lambda t}+\frac{ms}{4\lambda t}=
\frac{ms}{4\lambda}$,
as required.
In this case, 
$$\begin{array}{rcl}
\bigcup\limits_{x \in X} \supp(B\pm x)  & =& \bigcup\limits_{i=0}^{\frac{t-3}{2}}
\left( [i\ell+1, (i+1)\ell ]\setminus
\{(i+1)\ell \} \right)\cup \left[\left(\frac{t-1}{2}\right)\ell+1,
\left(\frac{t-1}{2}\right)\ell+\frac{\ell-1}{2} \right] \\[8pt]
 & = & \left(\left[1, \frac{t-1}{2}\ell \right]\setminus \left\{\ell, 2\ell, \ldots,
\frac{t-1}{2}\ell\right\}\right)\cup
\left[\left(\frac{t-1}{2}\right)\ell+1, \frac{ms}{\lambda}+\frac{t-1}{2} \right]
\\[8pt]
& =& \left[1,\frac{ms}{\lambda}+\left\lfloor\frac{t}{2}\right\rfloor\right]\setminus
\left\{\ell, 2\ell, \ldots, \left\lfloor\frac{t}{2}\right\rfloor\ell\right\}.
\end{array}$$

In both cases, we construct the p.f. array $A$ using  $\lambda$ copies of every block $B\pm x$;  so, the p.f. array $A$ obtained following our procedure is an integer 
${}^\lambda\H_t(m,n;s,k)$.
\end{proof}

For instance, we can follow the proof of the previous lemma for constructing an integer ${}^3\H_{3}(9,9;8,8)$.
In fact, $\lambda=3$ and $t=3$ divides $\frac{9\cdot 8}{2\cdot 3}$; note that $\ell=17$ and $Y=3\ast (0,4,8,12,17,21)$.

\begin{footnotesize}
$${}^3\H_{3}(9,9;8,8)=\begin{array}{|c|c|c|c|c|c|c|c|c|}\hline
\omb 1 & \omb -2 & -20  & 21  &13  &-14  &  &-7  &8  \\\hline
12 & 5 & -6 & -24& 25& 18& -19& & -11 \\\hline
\omb -3 & \omb 4 & 9  & -10  &-15 & 16 & 22 &-23  &    \\\hline
& -7& 8& 13& -14& -20& 21& 1& -2 \\\hline
 -6& & -11& 12& 18& -19& -24& 25& 5 \\\hline
9& -10& & -15& 16& 22& -23& -3& 4 \\\hline
8& 13& -14& & -20& 21& 1& -2& -7 \\\hline
-11& 12& 18& -19& & -24& 25& 5& -6 \\\hline
-10& -15& 16& 22& -23& & -3& 4& 9 \\\hline
  \end{array}\;.$$
\end{footnotesize}

We now consider the case when $\lambda$ does not divide $ms$. We need to adjust our general strategy in order to 
satisfy \eqref{cond2ms}.

\begin{lem}\label{lambda ms}
Suppose that $\lambda$ does not divide $ms$. There  exists an integer ${}^\lambda\H_t(m,n;s,k)$ for any divisor $t$ 
of $\frac{2ms}{\lambda}$.
\end{lem}

\begin{proof}
Since $\lambda$ divides $2ms$ but  does not divide $ms$, 
from $s\equiv 0 \pmod{4}$ we obtain $\lambda \equiv 0 \pmod 8$.
We can easily adapt the proof of Lemma 	\ref{lambda4}, using the block  $B=B_{0,0}=\begin{array}{|c|c|}\hline
     1 & -1 \\\hline
      & \\\hline
   -1 & 1\\\hline
    \end{array}$ and considering two possibilities.
In both cases,  an integer ${}^\lambda\H_t(m,n;s,k)$, say $A$, can be obtained following the construction given at 
the beginning of this section and using the blocks $B\pm y_0, B \pm y_1,\ldots,B\pm y_{\frac{ms}{4}-1}$ for a suitable sequence 
$Y=(y_0,y_1,\ldots,  y_{\frac{ms}{4}-1})$ in such a way that condition \eqref{cond2ms} is satisfied.

Suppose that $\ell$ is odd or $t$ is even.
It suffices to consider the sequence $X$ obtained by taking the natural ordering $\leq$ of $\{i-1\mid i\in \Phi \}\subset \N$,
and define $Y=\frac{\lambda}{4}\ast X$.

Suppose that $\ell$ is even and $t$ is odd.
Let $X_1$ be the sequence obtained by taking the natural ordering $\leq$ of $\{i-1\mid i\in \Psi \}\subset \N$,
where $\Psi= \Phi \setminus \left\{\frac{t\ell}{2}\right\}$.
Also, let $Y_1=\frac{\lambda}{4}\ast X_1$ and let $Y_2$ be the sequence obtained by repeating 
$\frac{\lambda}{8}$ times
the integer $\frac{t\ell}{2}-1$.
Define $Y=Y_1\con Y_2$ and note that $|Y|=\frac{\lambda}{4}\cdot \frac{2ms-\lambda}{2\lambda}+\frac{\lambda}{8}=\frac{ms}{4}$.
\end{proof}

For instance, we can follow the proof of the previous lemma for constructing an integer ${}^{16}\H_{5}(10,10;4,4)$.
In fact, $\lambda=16$ does not divide $ms=40$; note that $\ell=2$,
$X_1=(0,2)$ and $Y=(0,2,0,2,0,2,0,2, 4,4)$.

\begin{footnotesize}
$${}^{16}\H_{5}(10,10;4,4)=\begin{array}{|c|c|c|c|c |c|c|c|c|c|}\hline
   1 &-1 & & & & & & &-5 & 5 \\ \hline
    5 & 3 &-3 & & & & & & &-5 \\ \hline
   -1 & 1 & 1 & -1 & & & & & &  \\ \hline
    & -3 & 3 & 3 &-3 & & & & &  \\ \hline
    & & -1 & 1 & 1 &-1 & & & &  \\ \hline
    & & & -3 & 3 & 3 &-3 & & &  \\ \hline
    & & & &-1 & 1 & 1 &-1 & &  \\ \hline
    & & & & &-3 & 3 & 3 &-3 &  \\ \hline
    & & & & & &-1 & 1 & 5 &-5 \\ \hline
   -5 & & & & & & &-3 & 3 & 5 \\ \hline
  \end{array}\;.$$
\end{footnotesize}

\begin{prop}\label{prop:k4}
Suppose $4\leq s \leq n$, $4\leq k \leq m$, $ms=nk$ and $s,k\equiv 0 \pmod 4$. Let $\lambda$ be a divisor of $2ms$.
There exists a shiftable integer ${}^\lambda\H_t(m,n;s,k)$ for every divisor $t$ of $\frac{2ms}{\lambda}.$
\end{prop}

\begin{proof}
If $\lambda$ does not divide $ms$,  the statement follows from Lemma \ref{lambda ms}.
So, suppose that $\lambda$ divides $ms$.
If $\lambda\equiv 0 \pmod 4$ or $\lambda\equiv 2 \pmod 4$, then we can apply Lemma \ref{lambda4} or Lemma \ref{lambda2}, respectively.
Now we assume $\lambda$ odd.
If $t \equiv 0 \pmod 8$, we apply Lemma \ref{k4-3}.
If $t \equiv 4 \pmod 8$, then $t$ divides $\frac{ms}{\lambda}$ and hence we can apply Lemma \ref{k4-2}.
Finally, if $t\not \equiv 0 \pmod 4$, then $t$ divides $\frac{ms}{2\lambda}$ and so the existence of an  integer
${}^\lambda\H_t(m,n;s,k)$  follows from Lemma \ref{k4-1}.
In all these cases, the integer $\lambda$-fold Heffter array that we construct is
shiftable.
\end{proof}

\section{The case $s\equiv 2 \pmod 4$, $k$ and $m$ even}\label{s2}

In this section, we will assume that $s,m,k$ are positive even integers with $s \equiv 2 \pmod 4$ and $s\geq 6$.
We need to distinguish two cases, according to the divisibility of $ms$ by $\lambda$.
In fact, if $\lambda$ does not divide $ms$, from $ms\equiv 0 \pmod 4$ we obtain $\lambda\equiv 0 \pmod 8$.
In this case, we have to construct p.f. arrays that satisfy \eqref{cond2ms}.

If $\lambda$ divides $ms$ we write
\begin{equation}\label{lam}
\lambda=\lambda_1 \lambda_2,\quad \textrm{ where } \lambda_1 \textrm{ divides } \frac{m}{2} \textrm{ and }
\lambda_2\textrm{ divides } 2s.
\end{equation}
Let $t$ be a divisor of $\frac{2ms}{\lambda}$ and set
$$\ell=\frac{2ms}{\lambda t}+1.$$

\subsection{Construction of nice pairs of sequences}

To obtain an integer ${}^\lambda \H_t(m,n;s,k)$, we first construct pairs of sequences, satisfying the
following properties.

\begin{defi}
A pair $(\B_1,\B_2)$ of sequences is said to be \emph{nice} if, for a fixed positive integer $b$, we have:
\begin{itemize}
\item the sequence $\B_1$ consists of blocks satisfying this condition:
\begin{equation}\label{blocchi}
\begin{array}{l}
\textrm{there exist } b \textrm{ integers } \sigma_1,\ldots,\sigma_{b} \textrm{ such that the elements of } \B_1  \\
\textrm{are shiftable blocks } B \textrm{ of size }
2\times 2b \textrm{ with }\tau_1(B)=\tau_2(B)=0\\
\textrm{and } \gamma_{2i-1}(B)=-\gamma_{2i}(B)=\sigma_i \textrm{ for all } i \in [1,b];
\end{array}
\end{equation}
\item the sequence $\B_2$ consists of blocks satisfying this condition:
\begin{equation}\label{blocchiOLD}
\begin{array}{l}
\textrm{there exist } 2b \textrm{ integers } \sigma'_1,\ldots,\sigma'_{2b} \textrm{ with }
\sum\limits_{i=1}^{b} \sigma'_{2i-1} = \sum\limits_{i=1}^{b} \sigma'_{2i} =0,\\
\textrm{such that the elements of } \B_2 \textrm{ are shiftable blocks } B' \textrm{ of size } 2\times 2b\\
\textrm{with } \tau_1(B')=\tau_2(B')=0 \textrm{ and } \gamma_{i}(B')=\sigma'_i \textrm{ for all } i \in [1,2b];
\end{array}
\end{equation}
\item  the sequences $\B_1$ and $\B_2$ have the same length and, writing $\B_1=(B_1,B_2,\ldots,B_{e})$
and $\B_2=(B'_1,B'_2,\ldots,B'_{e})$, then $\E(B_i)=\E(B_i')$ for all $i\in [1,e]$.
\end{itemize}
\end{defi}
Observe that the sequences $\B_1,\B_2$ in the previous definition do not need to be distinct.

We construct these nice pairs of sequences, starting with the case when 
$\lambda$ divides $ms$.
In particular, our sequences $\B_i$,  consisting of shiftable blocks of size $2\times s$,
are of  length $\frac{m}{2\lambda_1}$ and such that $\mu(\B_i)=\lambda_2$.
We begin with the case when $\lambda_2$ is odd.
Note that this implies that $\lambda_2$ divides $\frac{s}{2}$.

\begin{lem}\cite[Corollary 4.10 and Lemma 5.1]{MP}\label{seqB}
Let $a$ and $c$ be even integers with $a\geq 2$, $c\geq 6$ and $c \equiv 2 \pmod 4$.
Let $u$ be a divisor of $2ac$ and set $\rho=\frac{2ac}{u}+1$.
There exists a nice pair $(\tilde\B_1,\tilde\B_2)$ of sequences  of length $\frac{a}{2}$, where $\tilde \B_1$ and $\tilde \B_2$
consist of blocks of size $2\times c$, $\mu(\tilde\B_1)=\mu(\tilde\B_2)=1$ and
$$\supp(\tilde\B_1)=\supp(\tilde\B_2)=\left[ 1,ac+\left\lfloor u/2
\right\rfloor \right]
\setminus \left\{j\rho: j \in \left[1, \left\lfloor u/2 \right\rfloor \right]\right\}.$$
\end{lem}

\begin{cor}\label{cor odd}
Let $\lambda=\lambda_1\lambda_2$ be as in \eqref{lam}.  If $\lambda_2\neq \frac{s}{2}$ is odd, there exists a nice pair $(\B_1,\B_2)$ of sequences of length $\frac{m}{2\lambda_1}$,
where $\B_1$ and $\B_2$ consist of blocks of size $2\times s$,  $\mu(\B_1)=\mu(\B_2)=\lambda_2$ and
$$\supp(\B_1)=\supp(\B_2)=\left[ 1,\frac{ms}{\lambda}+\left\lfloor \frac{t}{2}
\right\rfloor \right]
\setminus \left\{\ell,2\ell,\ldots, \left\lfloor \frac{t}{2} \right\rfloor\ell \right\}=\Phi.$$
\end{cor}

\begin{proof}
Take $a=\frac{m}{\lambda_1}$, $c=\frac{s}{\lambda_2}$ and $u=t$.
Since $\lambda_1$ divides $\frac{m}{2}$, $a$ is a positive even integer;
since $\lambda_2\neq \frac{s}{2}$ is odd and divides $2s$, then $c$ is an even integer such that $c\geq 6$
and $c \equiv 2 \pmod 4$.
Note that $t$ divides $2ac=\frac{2ms}{\lambda_1\lambda_2}$ and $\rho=\frac{2ac}{t}+1=\frac{2ms}{\lambda t}+1=\ell$.
Hence, we can apply Lemma \ref{seqB} obtaining a nice pair $(\tilde\B_1,\tilde\B_2)$ of sequences of
length $\frac{m}{2\lambda_1}$ consisting of blocks of size $2\times \frac{s}{\lambda_2}$
such that $\mu(\tilde\B_1)=\mu(\tilde \B_2)=1$ and $\supp(\tilde \B_1)=\supp(\tilde \B_2)=\Phi$.
Now, replace every block $\tilde B$ of $\tilde \B_i$, $i=1,2$, with the block $B$ obtained juxtaposing $\lambda_2$ copies
of $\tilde B$. So, $B$ is a block of size $2\times s$ and $\mu(B)=\lambda_2$.
Call $\B_1,\B_2$ the two sequences so obtained. It follows that the pair $(\B_1,\B_2)$ satisfies the required properties.
\end{proof}

Now we consider the case when $\lambda_2=\frac{s}{2}$.

\begin{lem}\label{s/2}
Let $\lambda=\lambda_1\lambda_2$ be as in \eqref{lam} with $\lambda_2=\frac{s}{2}$.
There exists a nice pair $(\B_1,\B_2)$ of sequences of length $\frac{m}{2\lambda_1}$, where $\B_1$ and $\B_2$
consist of  blocks of size $2\times s$,  $\mu(\B_1)=\mu(\B_2)=\frac{s}{2}$ and
$\supp(\B_1)=\supp(\B_2)=\Phi$.
\end{lem}

\begin{proof}
We first consider the case when $\ell$ is odd.
Consider the following shiftable blocks:

\begin{footnotesize}
$$\begin{array}{rclcrcl}
   A& =& \begin{array}{|c|c|c|c|}\hline
  1 & -2 & -3 & 4 \\\hline
 -1 &  2 & 3 & -4  \\ \hline
      \end{array}\;, & \quad &
       F& =& \begin{array}{|c|c|c|c|}\hline
  1 & -2 & -4 & 5 \\\hline
 -1 &  2 & 4 & -5  \\ \hline
      \end{array}\;,\\[8pt]
E & =& \begin{array}{|c|c|c|c|c|c|}\hline
  1 & -1 &  3 & -4 & -3 &  4 \\\hline
 -2 &  2 & -1 &  2 &  3 & -4 \\\hline
\end{array}\;,&&
G & =& \begin{array}{|c|c|c|c|c|c|}\hline
  4 &  2 & -2 &  2 & -1 & -5 \\\hline
 -5 & -1 &  4 & -4 &  1 &  5   \\\hline
\end{array}\;,\\[8pt]
E' & =& \begin{array}{|c|c|c|c|c|c|}\hline
  1 &  3 & -1  & -4 & -3 &  4 \\\hline
 -2 & -1 &  2  &  2 &  3 & -4 \\\hline
\end{array}\;,&&
G' & =& \begin{array}{|c|c|c|c|c|c|}\hline
  4 &  -2 &  2 &  2 & -1 & -5 \\\hline
 -5 &   4 & -1 & -4 &  1 &  5   \\\hline
\end{array}\;.
\end{array}$$
\end{footnotesize}

Note that $A$ and $F$ satisfy both \eqref{blocchi} and \eqref{blocchiOLD}; $E$ and $G$ satisfy \eqref{blocchi};
$E'$ and $G'$ satisfy \eqref{blocchiOLD}. We first construct the sequence $\B_1$.
To this purpose, take the block $B$ obtained juxtaposing the block $E$ and $\frac{s-6}{4}$ copies of the block $A$.
We obtain a block of size $2\times s$ such that $\supp(B)=[1,4]$ and $\mu(B)=\frac{s}{2}$.
Also, let $C$ be the block obtained juxtaposing the block $G$ and $\frac{s-6}{4}$ copies of the block $F$.
Then $C$ is a block of size $2\times s$ such that $\supp(C)=\{1,2,4,5 \}$ and $\mu(C)=\frac{s}{2}$.

Assume $\ell\equiv 1 \pmod 4$. Let $S=\left(B,B\pm 4, B \pm 8,\ldots,B\pm  4\frac{\ell-5}{4} \right)$.
Then $|S|=\frac{\ell-1}{4}$ and $\supp(S)=[1,\ell]\setminus\{\ell\}$.
If $t$ is even, take
$$\B_1=S\con (S\pm \ell)\con (S \pm 2 \ell)\con \ldots\con \left(S \pm \frac{t-2}{2}\ell \right).$$
If $t$ is odd, then $\ell-1=8\frac{m}{2\lambda_1 t}\equiv 0 \pmod 8$.
Let
$$Y=\left(B, B\pm 4 ,B\pm 8 ,\ldots,
B\pm \left(4\frac{\ell-9}{8}\right) \right)$$
and
$$\B_1=S\con (S\pm \ell)\con (S \pm 2 \ell)\con \ldots\con \left(S \pm \frac{t-3}{2}\ell \right)\con
\left(Y\pm \frac{t-1}{2}\ell\right).$$
In both cases, $\B_1$ is a sequence of length $\frac{(\ell-1)t}{8}=\frac{m}{2\lambda_1}$ such that
$\mu(\B_1)=\frac{s}{2}$ and $\supp(\B_1)=\Phi$.
The sequence $\B_2$ is obtained by replacing in $\B_1$ the block $E$ with the block $E'$.

Assume $\ell\equiv 3 \pmod 4$. Note that, in this case, $8\frac{m}{2\lambda_1 t}\equiv 2 \pmod 4$ and so $t\equiv 0\pmod 4$.
Let $S=\left(B,B\pm 4, B \pm 8,\ldots,B\pm  4\frac{\ell-7}{4},
C\pm (\ell-3), B\pm (\ell+2), B\pm (\ell+6), B\pm (\ell+10), \right.$ $\left. \ldots,B\pm (2\ell-5)\right)$.
Then $|S|=\frac{\ell-1}{2}$ and $\supp(S)=[1,2\ell]\setminus\{\ell,2\ell\}$.
Define
$$\B_1=S\con (S\pm 2\ell)\con (S \pm 4 \ell)\con \ldots\con \left(S \pm 2\frac{t-4}{4}\ell \right).$$
So, $\B_1$ is a sequence of length $\frac{(\ell-1)t}{8}=\frac{m}{2\lambda_1}$ such that
$\mu(\B_1)=\frac{s}{2}$ and $\supp(\B_1)=\Phi$.
The sequence $\B_2$ is obtained by replacing in $\B_1$ the block $G$ with the block $G'$.

Finally, assume that $\ell$ is even. Note that, in this case, $t\equiv 0\pmod 8$.
Consider the shiftable blocks:

\begin{footnotesize}
$$\begin{array}{rcl}
H& =& \begin{array}{|c|c|c|c|}\hline
  1 & -(\ell+1) & -(2\ell+1) &   3\ell+1 \\\hline
 -1 &   \ell+1  &   2\ell+1  & -(3\ell+1)  \\ \hline
      \end{array}\;, \\[8pt]
L & =& \begin{array}{|c|c|c|c|c|c|}\hline
  1        &  3\ell+1   &  -(\ell+1) &     \ell+1  & -1 & -(3\ell+1) \\\hline
 -(\ell+1) & -(2\ell+1) &    2\ell+1 &  -(2\ell+1) &  1 &   3\ell+1  \\\hline
\end{array}\;.
\end{array}$$
\end{footnotesize}

Note that the blocks $H$ and $L$ satisfy both \eqref{blocchi} and \eqref{blocchiOLD}.
Let $K$ be the block obtained juxtaposing the block $L$ and $\frac{s-6}{4}$ copies of the block $H$.
Then $K$ is a block of size $2\times s$ such that $\supp(K)=\{1,\ell+1,2\ell+1,3\ell+1 \}$ and $\mu(K)=\frac{s}{2}$.
Let $S=\left(K,K\pm 1, K \pm 2, \ldots, K\pm (\ell-2)\right)$.
Then $|S|=\ell-1$ and $\supp(S)=[1,4\ell]\setminus\{\ell,2\ell,3\ell,4\ell\}$.
Define
$$\B_1=\B_2=S\con (S\pm 4\ell)\con (S \pm 8 \ell)\con \ldots\con \left(S \pm 4\frac{t-8}{8}\ell \right).$$
So, $\B_i$ is a sequence of length $\frac{(\ell-1)t}{8}=\frac{m}{2\lambda_1}$ such that
$\mu(\B_i)=\frac{s}{2}$ and $\supp(\B_i)=\Phi$.
\end{proof}

For instance, taking in the previous lemma, $m=30$, $s=10$, $\lambda_1=3$ and $t=5$,
we have $\ell=9$. The sequence $\B_1$ consists of the following
five shiftable blocks:

\begin{footnotesize}
 $$\begin{array}{rcl}
B_1 & =& \begin{array}{|c|c|c|c|c|c| c|c|c|c|}\hline
 1 & -1 &  3 & -4 & -3 &  4 &  1 & -2 & -3 &  4 \\ \hline
-2 &  2 & -1 &  2 &  3 & -4 & -1 &  2 &  3 & -4 \\ \hline
\end{array}\;,\\[8pt]
B_2 & = &
 \begin{array}{|c|c|c|c|c|c| c|c|c|c|c|c| c|c|c|c|c|c|}\hline
 5 & -5 &  7 & -8 & -7 &  8 &  5 & -6 & -7 &  8 \\ \hline
-6 &  6 & -5 &  6 &  7 & -8 & -5 &  6 &  7 & -8 \\ \hline
\end{array}\;, \\[8pt]
B_3 & =& \begin{array}{|c|c|c|c|c|c| c|c|c|c|}\hline
 10 & -10 &  12 & -13 & -12 &  13 &  10 & -11 & -12 &  13 \\ \hline
-11 &  11 & -10 &  11 &  12 & -13 & -10 &  11 &  12 & -13 \\ \hline
\end{array}\;,\\[8pt]
B_4 & = &
 \begin{array}{|c|c|c|c|c|c| c|c|c|c|c|c| c|c|c|c|c|c|}\hline
 14 & -14 &  16 & -17 & -16 &  17 &  14 & -15 & -16 &  17 \\ \hline
-15 &  15 & -14 &  15 &  16 & -17 & -14 &  15 &  16 & -17 \\ \hline
\end{array}\;, \\[8pt]
B_5 & = &
 \begin{array}{|c|c|c|c|c|c| c|c|c|c|c|c| c|c|c|c|c|c|}\hline
 19 & -19 &  21 & -22 & -21 &  22 &  19 & -20 & -21 &  22 \\ \hline
-20 &  20 & -19 &  20 &  21 & -22 & -19 &  20 &  21 & -22 \\ \hline
\end{array}\;.
\end{array}$$
\end{footnotesize}

We now deal with the case $\lambda_2\equiv 2\pmod 4$.

\begin{lem}\label{magg 6}
Let $\lambda=\lambda_1\lambda_2$ be as in \eqref{lam} with $\lambda_2\equiv 2\pmod 4$ and $\lambda_2\geq 6$.
There exists a nice pair $(\B,\B)$, where $\B$ is a sequence of length $\frac{m}{2\lambda_1}$
consisting of  blocks of size $2\times s$ such that $\mu(\B)=\lambda_2$ and
$\supp(\B)=\Phi$.
\end{lem}

\begin{proof}
We first consider the case when $\ell$ is odd.
Consider the following shiftable blocks:

\begin{footnotesize}
$$\begin{array}{rclcrcl}
   A& =& \begin{array}{|c|c|c|c|}\hline
  1 & -1 & 2 & -2 \\\hline
 -1 &  1 & -2 & 2  \\ \hline
      \end{array}\;, & \quad &
E & =& \begin{array}{|c|c|c|c|c|c|}\hline
  1 &  2 & -1 &  1 & -1 & -2 \\\hline
 -2 & -1 &  2 & -2 &  1 &  2 \\\hline
\end{array}\;.
\end{array}$$
\end{footnotesize}

Note that $A$ and $E$ satisfy both \eqref{blocchi} and \eqref{blocchiOLD}.
To construct the sequence $\B$, first take the block $C$ obtained juxtaposing the block $E$ and $\frac{\lambda_2-6}{4}$
copies of the block $A$.
We obtain a block of size $2\times \lambda_2$ such that $\supp(C)=\{1,2\}$ and $\mu(C)=\lambda_2$.
Consider the sequence  $S=\left(C, C\pm 2, C\pm 4,\ldots, C\pm  2\frac{\ell-3}{2} \right)$.
Then $|S|=\frac{\ell-1}{2}$, $\mu(S)=\lambda_2$  and $\supp(S)=[1,\ell]\setminus\{\ell\}$.
If $t$ is even, take
$$\tilde\B=S\con (S\pm \ell)\con (S \pm 2 \ell)\con \ldots\con \left(S \pm \frac{t-2}{2}\ell \right).$$
If $t$ is odd, then $\ell-1=4\frac{\frac{m}{2\lambda_1}\cdot \frac{s}{\lambda_2}}{ t}\equiv 0 \pmod 4$.
Let
$$Y=\left(C, C\pm 2, C\pm 4,\ldots,
C\pm \left(2\frac{\ell-5}{4}\right) \right)$$
and
$$\tilde \B= S \con (S\pm \ell)\con (S \pm 2 \ell)\con \ldots\con \left(S \pm \frac{t-3}{2}\ell \right)\con
\left(Y\pm \frac{t-1}{2}\ell\right).$$
In both cases, $\tilde \B$ is a sequence of length $\frac{(\ell-1)t}{4}=\frac{ms}{2\lambda}$ such that
$\mu(\tilde \B)=\lambda_2$ and $\supp(\tilde\B)=\Phi$.

Suppose now that  $\ell$ is even.
Note that, in this case, $t\equiv 0\pmod 4$.
Consider the shiftable blocks:

\begin{footnotesize}
$$\begin{array}{rcl}
F& =& \begin{array}{|c|c|c|c|}\hline
  1 & -1 &    \ell+1   &  -(\ell+1) \\\hline
 -1 &  1  & -(\ell+1)  &    \ell+1  \\ \hline
      \end{array}\;, \\[8pt]
G & =& \begin{array}{|c|c|c|c|c|c|}\hline
  1        &   \ell+1   &    -1   &      1      &  -1 &  -(\ell+1) \\\hline
 -(\ell+1) &    -1      &  \ell+1 &   -(\ell+1) &   1 &    \ell+1  \\\hline
\end{array}\;.
\end{array}$$
\end{footnotesize}

Note that the blocks $F$ and $G$ satisfy both \eqref{blocchi} and \eqref{blocchiOLD}.
Take the block $H$ obtained juxtaposing the block $G$ and $\frac{\lambda_2-6}{4}$
copies of the block $F$.
We obtain a block of size $2\times \lambda_2$ such that $\supp(H)=\{1,\ell+1\}$ and $\mu(H)=\lambda_2$.
Consider the sequence  $S=\left(H, H\pm 1, H\pm 2,\ldots, H\pm\right.$ $\left. (\ell-2) \right)$.
Then $|S|=\ell-1$, $\mu(S)=\lambda_2$  and $\supp(S)=[1,2\ell]\setminus\{\ell,2\ell\}$.
Take
$$\tilde\B=S\con (S\pm 2\ell)\con (S \pm 4 \ell)\con \ldots\con \left(S \pm 2\frac{t-4}{4}\ell \right).$$
Hence, $\tilde \B$ is a sequence of length $\frac{(\ell-1)t}{4}=\frac{ms}{2\lambda}$ such that
$\mu(\tilde \B)=\lambda_2$ and $\supp(\tilde\B)=\Phi$.

Finally, for every $\ell$, writing $\tilde \B=\left(K_1,K_2,\ldots,K_{\frac{ms}{2\lambda}} \right)$
and $q=\frac{s}{\lambda_2}$, for every $i \in \left[1, \frac{m}{2\lambda_1}\right]$ we construct
the block $B_i$ juxtaposing
the $q$ blocks $K_{1+(i-1)q},K_{2+(i-1)q},\ldots,K_{iq}$. 
The blocks $B_i$ are of size $2\times q\lambda_2$, that is,
of size $2\times s$.
So, we can set $\B=\left(B_1,B_2,B_3,\ldots,B_{\frac{m}{2\lambda_1}} \right)$.
\end{proof}

For instance, taking in the previous lemma, $m=84$, $s=10$, $\lambda_1=7$, $\lambda_2=10$ and $t=8$,
we have $\ell=4$. The sequence $\B$ consists of the following
six shiftable blocks:

\begin{footnotesize}
 $$\begin{array}{rcl}
B_1 & =& \begin{array}{|c|c|c|c|c|c| c|c|c|c|}\hline
 1 &  5 & -1 &  1 & -1 & -5 &  1 & -1 &  5 & -5 \\ \hline
-5 & -1 &  5 & -5 &  1 &  5 & -1 &  1 & -5 &  5 \\ \hline
\end{array}\;,\\[8pt]
B_2 & = &
 \begin{array}{|c|c|c|c|c|c| c|c|c|c|c|c| c|c|c|c|c|c|}\hline
 2 &  6 & -2 &  2 & -2 & -6 &  2 & -2 &  6 & -6 \\ \hline
-6 & -2 &  6 & -6 &  2 &  6 & -2 &  2 & -6 &  6 \\ \hline
\end{array}\;, \\[8pt]
B_3 & = &
 \begin{array}{|c|c|c|c|c|c| c|c|c|c|c|c| c|c|c|c|c|c|}\hline
 3 &  7 & -3 &  3 & -3 & -7 &  3 & -3 &  7 & -7 \\ \hline
-7 & -3 &  7 & -7 &  3 &  7 & -3 &  3 & -7 &  7 \\ \hline
\end{array}\;, \\[8pt]
B_4 & = &
 \begin{array}{|c|c|c|c|c|c| c|c|c|c|c|c| c|c|c|c|c|c|}\hline
 9 &  13 &  -9 &   9 & -9 & -13 &  9 & -9 &  13 & -13 \\ \hline
-13 & -9 &  13 & -13 &  9 &  13 & -9 &  9 & -13 &  13 \\ \hline
\end{array}\;, \\[8pt]
B_5 & = &
 \begin{array}{|c|c|c|c|c|c| c|c|c|c|c|c| c|c|c|c|c|c|}\hline
 10 &  14 & -10 &  10 & -10 & -14 &  10 & -10 &  14 & -14 \\ \hline
-14 & -10 &  14 & -14 &  10 &  14 & -10 &  10 & -14 &  14 \\ \hline
\end{array}\;, \\[8pt]
B_6 & = &
 \begin{array}{|c|c|c|c|c|c| c|c|c|c|c|c| c|c|c|c|c|c|}\hline
 11 &  15 & -11 &  11 & -11 & -15 &  11 & -11 &  15 & -15 \\ \hline
-15 & -11 &  15 & -15 &  11 &  15 & -11 &  11 & -15 &  15 \\ \hline
\end{array}\;.
\end{array}$$
\end{footnotesize}

We now deal with the case $\lambda_2=2$.

\begin{lem}\label{2odd}
Let $\lambda=\lambda_1\lambda_2$ be as in \eqref{lam} with $\lambda_2=2$.
Suppose that $t$ divides $\frac{ms}{2\lambda_1}$.
There exists a nice pair $(\B_1,\B_2)$ of sequences of length $\frac{m}{2\lambda_1}$, where $\B_1$ and $\B_2$
consist of  blocks of size $2\times s$, $\mu(\B_1)=\mu(\B_2)=2$ and
$\supp(\B_1)=\supp(\B_2)=\Phi$.
\end{lem}

\begin{proof}
Write $s=4q+6$ where $q\geq 0$ and take the following shiftable blocks:

\begin{footnotesize}
$$\begin{array}{rclcrcl}
 U_3& =& \begin{array}{|c|c|c|c|}\hline
  1 & -2 & -4 &  5 \\\hline
 -1 &  2 &  4 & -5  \\ \hline
      \end{array}\;, & \quad &
U_5 & =& \begin{array}{|c|c|c|c|}\hline
  1 &  -2 & -3 &  4  \\\hline
 -1 &   2 &  3 & -4  \\\hline
\end{array}\;,\\[8pt]
 V_1& =& \begin{array}{|c|c|c|c|c|c|}\hline
  2 & -2 & -5 &  -6 &  4 &  7 \\\hline
 -3 &  3 &  6 &   5 & -4 & -7 \\ \hline
      \end{array}\;, & \quad &
  V_3& =& \begin{array}{|c|c|c|c|c|c|}\hline
  1 & -1 & -5 &  -6 &  4 &  7 \\\hline
 -2 &  2 &  6 &   5 & -4 & -7 \\ \hline
      \end{array}\;, \\[8pt]
 V_5& =& \begin{array}{|c|c|c|c|c|c|}\hline
  6 & -6 & -2 &  -3 &  1 &  4 \\\hline
 -7 &  7 &  3 &   2 & -1 & -4 \\ \hline
      \end{array}\;, & \quad &
  V_7& =& \begin{array}{|c|c|c|c|c|c|}\hline
  1 & -1 & -4 &  -5 &  3 &  6 \\\hline
 -2 &  2 &  5 &   4 & -3 & -6 \\ \hline
      \end{array}\;, \\[8pt]
 Z & =& \begin{array}{|c|c|c|c|c|c|}\hline
  1 & -1 &  4 &  -5 &  -7  &  8 \\\hline
 -2 &  2 & -4 &   5 &   7  & -8 \\ \hline
      \end{array}\;, & \quad &
  Z' & =& \begin{array}{|c|c|c|c|c|c|}\hline
  1 &  4  &  -1 &  -5 &  -7  &  8 \\\hline
 -2 &  -4 &   2 &   5 &   7  & -8 \\ \hline
      \end{array}\;.
\end{array}$$
\end{footnotesize}

Note that, since $t$ divides $\frac{ms}{2\lambda_1}$, $\ell$ is an odd integer.

If $\ell=4x+1\geq 5$, take
$\tilde S=(U_5, U_5\pm 4, U_5\pm 8,\ldots,U_5\pm 4(x-1) )$.
Then $|\tilde S|=x$, $\mu(\tilde S)=2$ and $\supp(\tilde S)=[1,\ell]\setminus\{\ell\}$.
Let $\tilde \B$ be the sequence obtained by taking the first $\frac{mq}{2\lambda_1}$ blocks in
$\con\limits_{c\geq 0} (\tilde S\pm\ell c)$.
If $\ell=4x+3\geq 3$, take
$\tilde S=(U_5, U_5\pm 4, U_5\pm 8,\ldots,U_5\pm 4(x-1),U_3\pm 4x, U_5\pm (4x+5), U_5\pm (4x+9),\ldots,
U_5\pm (8x+1))$.
Then $|\tilde S|=2x+1$, $\mu(\tilde S)=2$ and $\supp(\tilde S)=[1,2\ell]\setminus\{\ell,2\ell\}$.
Let $\tilde \B$ be the sequence obtained taking the first $\frac{mq}{2\lambda_1}$ blocks in
$\con\limits_{c\geq 0} (\tilde S\pm 2\ell c)$.
In both cases we obtain a sequence $\tilde \B$
of blocks of size $2\times 4$ that satisfy both \eqref{blocchi} and \eqref{blocchiOLD}
and such that $\supp(\tilde \B)=[1,N]$ where
$N=\frac{2mq}{\lambda_1}+\eta$ with $\eta=\left\lfloor \frac{2qt}{s}\right\rfloor$.

Now, we have to construct a sequence $S'$ of shiftable blocks of size $2\times 6$ satisfying condition \eqref{blocchi}
in such a way that $|S'|=\frac{m}{2\lambda_1}$ and
$$\supp(S')=\left[N+1,\frac{ms}{2\lambda_1}+\left\lfloor \frac{t}{2}\right\rfloor\right]
\setminus \left\{j\ell: j \in \left[\eta+1,\left\lfloor \frac{t}{2} \right\rfloor \right]\right\}.$$

If $\ell=3$, then $t=\frac{ms}{2\lambda_1}$ and $N=3\frac{mq}{\lambda_1}\equiv 0 \pmod 3$.
We can take
$S'=\con\limits_{c=0}^{\frac{m}{2\lambda_1}-1}(Z\pm (N+9c))$.
If $\ell=5$,  then $t=\frac{ms}{4\lambda_1}$ and $N=5\frac{mq}{2\lambda_1}\equiv 0 \pmod 5$.
Define $T=( V_5, V_3\pm 7)$.
If $\frac{m}{2\lambda_1}$ is even, we can take
$S'=\con\limits_{c=0}^{\frac{m}{4\lambda_1}-1}(T\pm (N+15c) )$.
If $\frac{m}{2\lambda_1}$ is odd, we can take
$S'=\left(\con\limits_{c=0}^{\frac{m-6\lambda_1}{4\lambda_1}}( T\pm (N+15c)  )\right)\con
\left(V_5\pm \left(\frac{ms}{2\lambda_1}+\frac{t-15}{2}\right)\right)$.

Suppose now that $\ell\geq 7$: in this case, any set of $6$ consecutive integers contains at
most one multiple of $\ell$. We start considering the interval $[N+1,N+6]$
and the first multiple of $\ell$ belonging to the interval $[N+1,\frac{ms}{2\lambda_1}+\lfloor t/2\rfloor]$.
So, if $(\eta+1)\ell$ is an element of $[N+1,N+6]$ we take the block $V_{r}$ where $r$ must be chosen in such a way that
$\supp(V_r\pm N)$ does not contain $(\eta+1)\ell$.
Otherwise, we take the block $V_{7}$ and repeat this process considering the interval $[N+7,N+12]$.

It will be useful to  define,  for all $b\geq 1$,  the sequence
$$H(b)=(V_{7},V_{7}\pm 6, V_{7} \pm 12,\ldots, V_{7}\pm 6(b-1)).$$
Also, we set $H(0)$ to be the empty sequence: so, for all $b\geq 0$ the sequence $H(b)$ contains $b$ elements and
$\supp(H(b))=[1,6b]$.

Write $(\eta+1)\ell-N=6h_0+r_0$, where $0\leq r_0 < 6$, and define the sequence
$$S'_0=(H(h_0), V_{r_0}\pm 6h_0 ).$$
Note that $r_0$ is odd, since $\ell$ is odd and $(\eta+1)\ell-N \equiv (\eta+1)\ell+\eta \equiv 1\pmod 2$.
Furthermore, $\supp(S_0'\pm N)=[N+1, N+6h_0+7]\setminus \{ (\eta+1)\ell \}$.

Now, for all $j\in [1,\lfloor t/2 \rfloor - \eta]$, write
$\ell-7+r_{j-1}=6h_j +r_j$, where $0\leq  r_j < 6$, and define the sequence
$$S_j'=\left(H(h_j)\pm\left(7j+6\sum_{i=0}^{j-1} h_i   \right),
V_{r_j}\pm  \left(7j+6\sum_{i=0}^j h_i  \right)\right).$$
Note that $(\eta+j+1)\ell-N=6\sum_{i=0}^{j} h_i +7j+ r_j $ and
$$
\supp(S_j'\pm N) = \left[N+1+7j+6\sum_{i=0}^{j-1} h_i,\;  N+7(j+1)+6\sum_{i=0}^j h_i  \right]\setminus \{(\eta+j+1)\ell\}.
$$
The elements of $S'$ are the first $\frac{m}{2\lambda_1}$ blocks in
$\con\limits_{c=0}^{\lfloor t/2 \rfloor - \eta} (S'_c\pm N)$.

Finally, writing $\tilde \B=\left(A_1,\ldots,A_{\frac{mq}{2\lambda_1}}\right)$ and $S'=\left(G_1,\ldots,G_{\frac{m}{2\lambda_1}} \right)$,
for all $i=1,\ldots,\frac{m}{2\lambda_1}$, let $B_i$ be the block of size $2\times s$ obtained by juxtaposing the $q$ blocks
$$A_{(i-1)q+1},\;A_{(i-1)q+2},\;A_{(i-1)q+3},\; \ldots,A_{iq}$$
and the  block $G_i$.
By construction, the sequence $\B_1=(B_1,\ldots,B_{\frac{m}{2\lambda_1}})$ satisfies condition \eqref{blocchi},
has cardinality $\frac{m}{2\lambda_1}$, $\mu(\B_1)=2$ and
$\supp(\B_1)=\supp(S)\cup \supp(S')=\Phi$.

The sequence $\B_2$ is obtained by $\B_1$ replacing the block $Z$ with $Z'$ (case $\ell=3$).
\end{proof}

\begin{lem}\label{8p}
Let $\lambda=\lambda_1\lambda_2$ be as in \eqref{lam} with $\lambda_2=2$.
Let $p$ be an odd prime dividing $s$ and suppose
that $t$ is a divisor of $\frac{ms}{\lambda_1}$ such that $t\equiv 0 \pmod {4p}$.
There exists a nice pair $(\B,\B)$, where $\B$ is a sequence of length $\frac{m}{2\lambda_1}$ consisting of
blocks of size $2\times s$ such that $\mu(\B)=2$ and $\supp(\B)=\Phi$.
\end{lem}

\begin{proof}
Take the following blocks:

\begin{footnotesize}
$$\begin{array}{rcl}
W_4 & =& \begin{array}{|c|c|c|c|}\hline
       1 &   -(\ell+1) &  -(2\ell+1) &    3\ell+1\\\hline
      -1 &     \ell+1  &    2\ell+1   & -(3\ell+1)\\\hline
\end{array}\;,\\[8pt]
W_6 & = &
\begin{array}{|c|c|c|c|c|c|}\hline
         1 &         -1 &   -(3\ell+1) &      -(4\ell+1) &    2\ell+1  &      5\ell+1 \\\hline
 -(\ell+1) &     \ell+1 &     4\ell+1  &        3\ell+1  &  -(2\ell+1) &    -(5\ell+1) \\\hline
\end{array}\;.
\end{array}$$
\end{footnotesize}

\noindent Then $W_4$ and $W_6$ satisfy both properties \eqref{blocchi} and \eqref{blocchiOLD}
with column sums  $(0,0,0,0)$ and
 $(-\ell, \ell, \ell, -\ell,0,0)$, respectively.
 Furthermore, $\mu(W_4)=\mu(W_6)=2$ and
 $$\supp(W_4)=\{j\ell+1: j \in [0,3] \} \equad \supp(W_6)=\{j\ell+1: j \in [0,5]\}.$$
 Let $V$ be the following $2\times 2p$ block:
 $$V= \begin{array}{|c|c|c|c|c|}\hline
 W_6  & W_4\pm 6\ell & W_4\pm 10\ell & \cdots & W_4\pm (2p-4)\ell \\\hline
 \end{array}\;.$$
 Clearly, also $V$ satisfies both \eqref{blocchi} and \eqref{blocchiOLD} and its support is
 $\supp(V)=\{j\ell+1: j \in [0,2p-1] \}$.
We can use this block $V$ for constructing our sequence $\B$:
the $2\times s$ blocks of $\B$ are obtained simply by juxtaposing $h=\frac{s}{2p}$
blocks of type $V\pm x$, for $x\in X \subset \N$, following the natural order of $(X,\leq)$.
 So, we are left to exhibit  a suitable set $X$ of size $\frac{mh}{2\lambda_1}$
such that the support of the corresponding sequence $\B$ is $\Phi$.

 Let first $X_0=[0,\ell-2]$. Then $\supp(V\pm x_{i_1})\cap \supp(V\pm x_{i_2})=\emptyset$ for each $x_{i_1},x_{i_2}\in
X_0$ such that $x_{i_1}\neq x_{i_2}$. Furthermore,
 $$\bigcup_{x \in X_0} \supp(V\pm x)=[1,2p\ell]\setminus \{j\ell: j \in [1,2p] \}.$$
 Similarly, for any $i\in \N$, if  $X_i=[2pi\ell, (2pi+1)\ell-2]$ then
 $$\bigcup_{x \in X_i} \supp(V\pm x)=[1+2pi\ell, 2p\ell+2pi\ell]\setminus \{j\ell: j \in [1+2pi,2p+2pi] \}.$$
 Clearly, $X_{i_1}\cap X_{i_2}=\emptyset$ if $i_1\neq i_2$.
 Therefore, take
 $X=\bigcup\limits_{i=0}^{\frac{t}{4p}-1} X_i$: this is a set of size $\frac{t}{4p}\cdot (\ell-1)=\frac{t}{4p}\cdot
 \frac{4mph}{2\lambda_1 t}=\frac{mh}{2\lambda_1}$.
 It follows that the sequence $\B$ obtained, as previously described, from the blocks $V\pm x$, with $x\in X$, has
support equal to
 $$\begin{array}{rcl}
\supp(\B) & =&  \bigcup\limits_{i=0}^{\frac{t}{4p}-1}([1+2pi\ell,2p\ell+2pi\ell]\setminus \{j\ell: j \in [1+2pi,2p+2pi]
\})\\[8pt]
 & =& \left [1,\frac{t}{2}\ell \right ]\setminus \left \{j\ell: j \in \left [1,\frac{t}{2}\right ] \right \}
 =\left[1,\frac{ms}{2\lambda_1}+\frac{t}{2}\right]\setminus  \left\{\ell, 2\ell, \ldots, \frac{t}{2}\ell\right\},
 \end{array}$$
 as required.
\end{proof}

\begin{ex}\label{18x10}
Taking in the previous lemma, $m=18$, $s=10$, $\lambda_1=3$ and $t=20$,
we can choose $p=5$ so that $t\equiv 0\pmod{20}$. Hence $\ell=4$ and $\B$ consists of the following
three shiftable blocks:

\begin{footnotesize}
 $$\begin{array}{rcl}
B_1 & =& \begin{array}{|c|c|c|c|c|c|c|c|c|c|}\hline
  1 & -1 & -13 & -17 &  9 &  21 &  25 & -29 & -33 &  37 \\\hline
 -5 &  5 &  17 &  13 & -9 & -21 & -25 &  29 &  33 & -37 \\\hline
\end{array}\;,\\[8pt]
\end{array}$$
 $$\begin{array}{rcl}
B_2 & = &
\begin{array}{|c|c|c|c|c|c|c|c|c|c|}\hline
  2 & -2 & -14 & -18 &  10 &  22 &  26 & -30 & -34 &  38 \\ \hline
 -6 &  6 &  18 &  14 & -10 & -22 & -26 &  30 &  34 & -38 \\ \hline
\end{array}\;, \\[8pt]
B_3 & = &
\begin{array}{|c|c|c|c|c|c|c|c|c|c|}\hline
  3 & -3 & -15 & -19 &  11 &  23 &  27 & -31 & -35 &  39 \\ \hline
 -7 &  7 &  19 &  15 & -11 & -23 & -27 &  31 &  35 & -39 \\ \hline
\end{array}\;.
\end{array}$$
\end{footnotesize}
\end{ex}

\begin{lem}\label{non 8p}
Let $\lambda=\lambda_1\lambda_2$ be as in \eqref{lam} with $\lambda_2=2$.
Let $p$ be  an odd prime $p$ dividing $s$ and suppose
that $t$ is a divisor of $\frac{ms}{\lambda_1 p}$ such that $t\equiv 0 \pmod {4}$.
There exists a nice pair $(\B_1,\B_2)$ of sequences of length $\frac{m}{2\lambda_1}$, where
$\B_1$ and $\B_2$ consist of  blocks of size $2\times s$, $\mu(\B_1)=\mu(\B_2)=2$ and
$\supp(\B_1)=\supp(\B_2)=\Phi$.
\end{lem}

\begin{proof}
By hypothesis we can write $\ell=p y +1$.
Consider the following blocks:

\begin{footnotesize}
 $$\begin{array}{rcl}
W_4 & =& \begin{array}{|c|c|c|c|}\hline
    y+1  &  -(2y+1) &  -((p+1)y+2) &   (p+2)y+2\\\hline
  -(y+1) &    2y+1  &    (p+1)y+2 &  -((p+2)y+2) \\\hline
\end{array}\;,\\[8pt]
W_6 & = &
\begin{array}{|c|c|c|c|c|c|}\hline
   2y+1  &  -(2y+1) &  1  & -(y+1) & -((p+1)y+2) &   (p+2)y+2 \\\hline
 -(py+2) &    py+2  & -1  &   y+1  &   (p+1)y+2 & -((p+2)y+2)\\\hline
\end{array}\;, \\[8pt]
W_6' & = &
\begin{array}{|c|c|c|c|c|c|}\hline
   2y+1  &   1 &  -(2y+1) & -(y+1) & -((p+1)y+2) &   (p+2)y+2 \\\hline
 -(py+2) &  -1 &  py+2    &   y+1  &   (p+1)y+2 & -((p+2)y+2)\\\hline
\end{array}\;.
\end{array}$$
\end{footnotesize}

Note that the block $W_4$ satisfies both conditions \eqref{blocchi} and \eqref{blocchiOLD},
while $W_6$ satisfies condition \eqref{blocchi} and  $W_6'$ satisfies condition \eqref{blocchiOLD}.
Furthermore,
 $$\begin{array}{rcl}
\supp(W_4) & =& \{ (jp+1)y+j+1  , (jp+2)y+j+1: j \in [0,1] \},\\
\supp(W_6)=\supp(W_6') & =& \{ jpy+j+1, (jp+1)y+j+1, (jp+2)y+j+1: j \in [0,1]\}.
   \end{array}$$
Let $V$ be the following $2\times 2p$ block:
 $$V=
 \begin{array}{|c|c|c|c|c|}\hline
 W_6  & W_4\pm 2y & W_4\pm 4y & \cdots & W_4\pm (p-3)y \\\hline
 \end{array}\;.
 $$
 Clearly, $V$ satisfies \eqref{blocchi} and its support is
 $$\begin{array}{rcl}
\supp(V) & =&  \{iy+1,(p+i)y+2  :  i \in [0,p-1]\}\\
&= & \{iy+1,\ell+(iy+1) : i\in[0,p-1]\}.
   \end{array}$$
 We can use this block $V$ for constructing the sequence $\B_1$ as done in Lemma \ref{8p}:
 it suffices to exhibit  a suitable set $X$ of size
$\frac{mh}{2\lambda_1}$,  where $h=\frac{s}{2p}$,   such that the support of the corresponding sequence  $\B_1$ is $\Phi$.

 Let first $X_0=[0,y-1]$. Then $\supp(V\pm x_{i_1})\cap \supp(V\pm x_{i_2})=\emptyset$ for each $x_{i_1},x_{i_2}\in
X_0$ such that $x_{i_1}\neq x_{i_2}$. Furthermore,
 $$\bigcup_{x \in X_0} \supp(V\pm x)= [1,py]\cup [\ell+1,\ell+py] =[1,2\ell]\setminus \{\ell, 2\ell\}.$$
 Similarly, for any $i\in \N$, if  $X_i=[2i\ell , 2i\ell +y -1]$ then
 $$\bigcup_{x \in X_i} \supp(V\pm x)=[1+2i\ell, (2i+2)\ell ]\setminus \{(2i+1)\ell, (2i+2)\ell\}.$$
 Clearly, $X_{i_1}\cap X_{i_2}=\emptyset$ if $i_1\neq i_2$.
 Therefore, take
 $X=\bigcup\limits_{i=0}^{\frac{t}{4}-1} X_i$: this is a set of size $\frac{t}{4}\cdot y=\frac{t}{4}\cdot \frac{\ell  -1}{p}=
 \frac{t}{4} \cdot \frac{2mh}{\lambda_1 t}=\frac{mh}{2\lambda_1}$.
 It follows that the sequence $\B_1$ obtained  from the blocks $V\pm x$, with $x\in X$, has
support equal to
 $$\begin{array}{rcl}
\supp(\B_1) & =&  \bigcup\limits_{i=0}^{\frac{t}{4}-1}([1+2i\ell,2\ell(i+1)]\setminus \{(2i+1)\ell, (2i+2)\ell\})\\[8pt]
 & =& \left [1,\frac{t}{2}\ell \right ]\setminus \left \{\ell, 2\ell, \ldots, \frac{t}{2}\ell \right \}
 =\Phi,
 \end{array}$$
 as required. The sequence $\B_2$ is obtained using $W_6'$ instead of $W_6$.
\end{proof}

The last case we need is when $\lambda_2\equiv  0 \pmod 4$.

\begin{lem}\label{s 0}
Let $\lambda=\lambda_1\lambda_2$ be as in \eqref{lam} with  $\lambda_2\equiv 0 \pmod 4$.
There exists a nice pair $(\B,\B)$, where $\B$ is a sequence of length $\frac{m}{2\lambda_1}$
consisting of blocks of size $2\times s$ such that $\mu(\B)=\lambda_2$ and $\supp(\B)=\Phi$.
\end{lem}

\begin{proof}
Let $Q$ be the $2\times \frac{\lambda_2}{2}$ block obtained juxtaposing $\frac{\lambda_2}{4}$ copies of the shiftable block
$$\begin{array}{|c|c|}\hline
  1 & -1 \\\hline
 -1 & 1  \\ \hline
      \end{array}\;.$$
Clearly, $Q$ satisfies both conditions \eqref{blocchi} and \eqref{blocchiOLD}.
Furthermore, $\supp(Q)=\{1\}$ and $\mu(Q)=\lambda_2$.
Take a partition of $\Phi$ into $\frac{m}{2\lambda_1}$ subsets $X_i$, each of cardinality   $\frac{2s}{\lambda_2}$.
Writing, for all $i\in \left[1,\frac{m}{2\lambda_1}\right]$,
$X_i=\left\{x_{i,1},x_{i,2},\ldots,x_{i, \frac{2s}{\lambda_2}}\right\}$,
let $B_i$ the block
$$B_i=\begin{array}{|c|c|c|c|c|}\hline
Q \pm (x_{i,1}-1) &  Q\pm (x_{i,2}-1) &  Q\pm (x_{i,3}-1) & \cdots &
Q \pm \left(x_{i,\frac{2s}{\lambda_2}}-1\right) \\[-10pt]
&&&& \\\hline \end{array}.$$
Then each $B_i$ is a block of size $2\times s$ such that  $\supp(B_i)=X_i$
and $\mu(B_i)=\lambda_2$.
Finally, it suffices to take the sequence $\B=\left(B_1,B_2,\ldots,B_{\frac{m}{2\lambda_1}}\right)$.
\end{proof}

\begin{ex}\label{16x10}
Taking in the previous lemma, $m=16$, $s=10$, $\lambda_1=2$, $\lambda_2=4$ and $t=5$,
we have $\ell=9$ and $\Phi=[1,22]\setminus\{9,18\}$. So, can take
$X_1=[1,5]$, $X_2=[6,11]\setminus\{9\}$, $X_3=[12,16]$ and $X_4=[17,22]\setminus\{18\}$.
Hence, the sequence $\B$ consists of the following four shiftable blocks:

\begin{footnotesize}
 $$\begin{array}{rcl}
B_1 & =& \begin{array}{|c|c|c|c|c|c| c|c|c|c|}\hline
 1 & -1 &  2 & -2 &  3 & -3 &  4 & -4 &  5 & -5 \\ \hline
-1 &  1 & -2 &  2 & -3 &  3 & -4 &  4 & -5 &  5 \\ \hline
\end{array}\;,\\[8pt]
B_2 & = &
 \begin{array}{|c|c|c|c|c|c| c|c|c|c|c|c| c|c|c|c|c|c|}\hline
 6 & -6 &  7 & -7 &  8 & -8 &  10 & -10 &  11 & -11 \\ \hline
-6 &  6 & -7 &  7 & -8 &  8 & -10 &  10 & -11 &  11 \\ \hline
\end{array}\;, \\[8pt]
B_3 & = &
 \begin{array}{|c|c|c|c|c|c| c|c|c|c|c|c| c|c|c|c|c|c|}\hline
 12 & -12 &  13 & -13 &  14 & -14 &  15 & -15 &  16 & -16 \\ \hline
-12 &  12 & -13 &  13 & -14 &  14 & -15 &  15 & -16 &  16 \\ \hline
\end{array}\;, \\[8pt]
B_4 & = &
 \begin{array}{|c|c|c|c|c|c| c|c|c|c|c|c| c|c|c|c|c|c|}\hline
 17 & -17 &  19 & -19 &  20 & -20 &  21 & -21 &  22 & -22 \\ \hline
-17 &  17 & -19 &  19 & -20 &  20 & -21 &  21 & -22 &  22 \\ \hline
\end{array}\;.
\end{array}$$
\end{footnotesize}
\end{ex}

\begin{prop}\label{nice}
Suppose that $\lambda$ divides $ms$ and write  $\lambda=\lambda_1\lambda_2$ be as in \eqref{lam}.
There exists a nice pair $(\B_1,\B_2)$ of sequences of length $\frac{m}{2\lambda_1}$,
where $\B_1$ and $\B_2$ consist of blocks of size $2\times s$, $\mu(\B_1)=\mu(\B_2)=\lambda_2$ and
$$\supp(\B_1)=\supp(\B_2)=\left[ 1,\frac{ms}{\lambda}+\left\lfloor \frac{t}{2} \right\rfloor \right]
\setminus \left\{\ell,2\ell,\ldots \left\lfloor \frac{t}{2} \right\rfloor \ell \right\}=\Phi.$$
\end{prop}

\begin{proof}
If $\lambda_2=\frac{s}{2}$, the statement follows from Lemma \ref{s/2}.
If $\lambda_2\neq \frac{s}{2}$ is odd, we apply Corollary \ref{cor odd}.
If $\lambda_2\equiv 0 \pmod 4$, we use Lemma \ref{s 0}.
So, we may assume $\lambda_2\equiv 2 \pmod 4$.
If $\lambda_2\geq 6$, the statement follows from Lemma \ref{magg 6}.
Finally, suppose $\lambda_2=2$.
Since $s\geq 6$ and $s\equiv 2 \pmod 4$,
there exists  an odd prime $p$ that divides $s$.
Now, our analysis depends on $t$; recall that $t$ is a divisor of $\frac{ms}{\lambda_1}$.
If $t$ divides $\frac{ms}{2\lambda_1}$, we apply Lemma \ref{2odd}.
Otherwise, we must have $t\equiv 0 \pmod 4$.
If $t$ divides $\frac{ms}{\lambda_1 p}$, the result follows from Lemma \ref{non 8p}.
If $t$ does not divide $\frac{ms}{\lambda_1 p}$,
then $t$ is divisible by $p$. In particular, $t\equiv 0\pmod{4p}$ and so we can apply Lemma \ref{8p}.
\end{proof}

\begin{prop}\label{nice 2ms}
Suppose that $\lambda$ does not divide $ms$.
There exists a nice pair $(\B,\B)$, where $\B$ is a sequence of length $\frac{m}{2}$
consisting of blocks of size $2\times s$, such that 
$\supp(\B)=\Phi$ and condition \eqref{cond2ms} is satisfied.
\end{prop}

\begin{proof}
As previously observed, we have $\lambda\equiv 0\pmod 8$.
Let $Q$ be  the following shiftable block:
$$Q=\begin{array}{|c|c|}\hline
  1 & -1 \\\hline
 -1 & 1  \\ \hline
      \end{array}\;.$$
Clearly, $Q$ satisfies both conditions \eqref{blocchi} and \eqref{blocchiOLD}.
Furthermore, $\supp(Q)=\{1\}$ and $\mu(Q)=4$.

Suppose that $\ell$ is odd or $t$ is even.
Consider the sequence $X$ obtained by taking the natural ordering $\leq$ of $\{i-1\mid i\in \Phi \}\subset \N$,
and define $Y=\frac{\lambda}{4}\ast X$.

Suppose that $\ell$ is even and $t$ is odd.
Let $X_1$ be the sequence obtained by taking the natural ordering $\leq$ of $\{i-1\mid i\in \Psi \}\subset \N$,
where $\Psi= \Phi \setminus \left\{\frac{t\ell}{2}\right\}$.
Also, let $Y_1=\frac{\lambda}{4}\ast X_1$ and let $Y_2$ be the sequence obtained by repeating $\frac{\lambda}{8}$ times
the integer $\frac{t\ell}{2}-1$.
Define $Y=Y_1\con Y_2$ and note that $|Y|=\frac{ms}{4}$.

In both cases, write $Y=(y_1,y_2,\ldots,y_{\frac{ms}{4}})$.
For all $i\in \left[1,\frac{m}{2}\right]$,
let $B_i$ the block
$$B_i=\begin{array}{|c|c|c|c|}\hline
Q \pm y_{1+(i-1)\frac{s}{2}} &  Q\pm y_{2+(i-1)\frac{s}{2}}  &   \cdots &
Q \pm y_{i\frac{s}{2}} \\[-10pt]
&&&\\\hline \end{array}.$$
Then each $B_i$ is a block of size $2\times s$:
it suffices to take the sequence $\B=\left(B_1,B_2,\ldots,B_{\frac{m}{2}}\right)$.
\end{proof}

\subsection{The subcase $k\equiv 0\pmod 4$}

Assuming $k\equiv 0\pmod 4$, from $ms=nk$ it follows that $m$ must be even.
We now explain how to arrange the blocks of the sequences previously constructed, in order to build
an integer ${}^\lambda \H_t(m,n;s,k)$.
To this  purpose, we define a `base unit' that we will fill with the elements of the blocks.

Let $\G=(G_1,\ldots,G_d)$ be a sequence of blocks such that the following property is satisfied:
\begin{equation}\label{blocchi2}
\begin{array}{l}
\textrm{there exist } b \textrm{ integers } \sigma_1,\ldots,\sigma_b \textrm{ such that the elements of }
\G \textrm{ are blocks } G_r\\
 \textrm{of size } 2\times 2b \textrm{ with } \gamma_{2i-1}(G_r)=-\gamma_{2i}(G_r)=\sigma_i \textrm{ for all } i \in [1,b].
\end{array}
\end{equation}

So, let  $\G$ be a sequence satisfying \eqref{blocchi2}, where the blocks $G_{r}=(g_{i,j}^{(r)})$
are all of size $2\times 2b$, with $2b\leq d$.
Let $P=P(\G)$ be the p.f. array of size $2d\times d$ so defined.\label{bloccoP}
For all $i \in [1,b]$ and all $j \in [1,2b]$,
the cell $(i,  i+j-1)$ of $P$ is filled with the element $g_{1,j}^{(i)}$ and the cell $(d+i,i+j-1)$ is filled with
the element $g_{2,j}^{(i)}$; here, the  column indices are taken modulo $d$.
The remaining cells of $P$ are empty. An example of such construction is given in Figure \ref{P}.

\begin{figure}[!ht]
\begin{footnotesize}
$$\begin{array}{|c|c|c|c|c|c|}\hline
 g_{1,1}^{(1)} & g_{1,2}^{(1)} & g_{1,3}^{(1)} & g_{1,4}^{(1)} &   &    \\\hline
    &  g_{1,1}^{(2)} & g_{1,2}^{(2)} & g_{1,3}^{(2)} & g_{1,4}^{(2)} &   \\ \hline
    &   & g_{1,1}^{(3)} & g_{1,2}^{(3)} & g_{1,3}^{(3)} & g_{1,4}^{(3)}   \\ \hline
 g_{1,4}^{(4)}   &   &   & g_{1,1}^{(4)} & g_{1,2}^{(4)} & g_{1,3}^{(4)}  \\ \hline
 g_{1,3}^{(5)}  & g_{1,4}^{(5)}  &   &   & g_{1,1}^{(5)} & g_{1,2}^{(5)}  \\ \hline
 g_{1,2}^{(6)} & g_{1,3}^{(6)} & g_{1,4}^{(6)} &  & & g_{1,1}^{(6)}\\\hline
 g_{2,1}^{(1)} & g_{2,2}^{(1)} & g_{2,3}^{(1)} & g_{2,4}^{(1)} &   &    \\\hline
    &  g_{2,1}^{(2)} & g_{2,2}^{(2)} & g_{2,3}^{(2)} & g_{2,4}^{(2)} &   \\ \hline
    &   & g_{2,1}^{(3)} & g_{2,2}^{(3)} & g_{2,3}^{(3)} & g_{2,4}^{(3)}   \\ \hline
 g_{2,4}^{(4)}   &   &   & g_{2,1}^{(4)} & g_{2,2}^{(4)} & g_{2,3}^{(4)}  \\ \hline
 g_{2,3}^{(5)}  & g_{2,4}^{(5)}  &   &   & g_{2,1}^{(5)} & g_{2,2}^{(5)}  \\ \hline
 g_{2,2}^{(6)} & g_{2,3}^{(6)} & g_{2,4}^{(6)} &  & & g_{2,1}^{(6)}\\\hline
  \end{array}$$
  \end{footnotesize}
  \caption{This is a $P(G_1,\ldots,G_6)$, where $G_1,\ldots,G_6$ are arrays of size  $2\times 4$.}\label{P}
\end{figure}

We prove that $P$ is a p.f. array whose columns all sum to zero.
Observe that every row of $P$ contains exactly $2b$ filled cells and every column contains exactly $4b$ elements.
The elements of the  $i$-th column of $P$ are
$$g_{1,1}^{(i)}, g_{1,2}^{(i-1)}, \ldots, g_{1,2b}^{(i+1-2b)},\;
g_{2,1}^{(i)}, g_{2,2}^{(i-1)}, \ldots, g_{2,2b}^{(i+1-2b)},$$
where the exponents must be read modulo $d$, with residues in $[1,d]$.
Since the sequence $\G$ satisfies \eqref{blocchi2}, we obtain
$$\gamma_i(P)=  \sum_{j=1}^{2b} \gamma_j (G_{i+1-j}) =
\sum_{j=1}^{2b} \gamma_j ( G_{i}) =
\sum_{u=1}^b (\sigma_u-\sigma_u)=0.$$
Furthermore, notice that $\tau_j(P)=\tau_1(G_j)$ and $\tau_{d+j}(P)=\tau_2(G_j)$ for all $j \in [1,d]$.

\begin{prop}\label{prop s2}
Suppose  $4\leq s \leq n$, $4\leq k \leq m$ and $ms=nk$.
Let $\lambda$ be a divisor of $2ms$ and let $t$ be a divisor of  $\frac{2ms}{\lambda}$.
There exists a shiftable integer ${}^\lambda\H_t(m,n;s,k)$ in each of the following cases:
\begin{itemize}
\item[$(1)$] $s\equiv 2\pmod 4$ and $k\equiv 0 \pmod 4$;
\item[$(2)$] $s\equiv 0\pmod 4$ and $k\equiv 2 \pmod 4$.
\end{itemize}
\end{prop}

\begin{proof}
(1) If $\lambda$ divides $ms$, let $(\B_1,\B_2)$ be the nice pair of sequences constructed in
Proposition \ref{nice} and set $\B=\lambda_1\ast\B_1$.
If $\lambda$ does not divide $ms$, let $\B$ be the sequence constructed in Proposition \ref{nice 2ms}.
Write $d=\gcd(\frac{m}{2}, n)$ and $a=\frac{sd}{n}$.
Note that $a$ is even integer. In fact, write $m =2 \bar m d $ and $n=d \bar n$. Since $k \equiv 0 \pmod 4$,
from  $\frac{s}{2}\cdot \frac{m}{2}=n\frac{k}{4}$ we obtain that $\bar n$ divides $\frac{s}{2}$.

Given a block $B_h \in \B$, define for every $j\in [1,\bar n]$ the block
$T_j(B_h)$ of size $2\times a$ consisting of the columns $C_i$ of $B_h$ with $i\in [a(j-1)+1,a j]$.
So, the block $B_h$ of size $2\times s$ is obtained juxtaposing the blocks $T_1(B_h), T_2(B_h),\ldots,T_{\bar n}(B_h)$.
Furthermore, for all $i \in [1,\bar m]$ and all $j \in [1,\bar n]$, each of the sequences
$$\left(T_j(B_{(i-1)d+1}),T_j(B_{(i-1)d+2}),\ldots,T_j(B_{id}) \right),$$
of cardinality $d$, satisfies condition \eqref{blocchi2}.

Let $A$ be an empty array of size $\bar m\times \bar n$. For every $i \in [1,\bar m]$ and  $j \in [1,\bar n]$,
replace the cell $(i,j)$ of $A$
with the block $P\left(T_j(B_{(i-1)d+1}),T_j(B_{(i-1)d+2}),\ldots,\right.$ $\left. T_j(B_{id}) \right)$,
according to the previous definition.
Note that, for all $r\in [1,\frac{m}{2}]$, we have $\tau_r(A)=\tau_1(B_r)=0$ and  $\tau_{r+\frac{m}{2}}(A)=\tau_2(B_r)=0$.

By construction, $A$ is a p.f. array of size $m\times n$, $\supp(A)=\Phi$ and  the rows and columns of $A$ sum to zero.
If $\lambda$ divides $ms$, then every element of $\Phi$ appears, up to sign, exactly
$\lambda$ times. 
If $\lambda$ does not divide $ms$, condition \eqref{cond2ms} is satisfied.
Furthermore, each row contains  $a\bar n=s$ elements and each column contains $2a\bar m=k$ elements.
We conclude that $A$ is a shiftable integer ${}^\lambda\H_t(m,n;s,k)$.\\
(2) It follows from (1). In fact, if $s\equiv 0 \pmod 4$ and $k\equiv 2\pmod 4$,
an integer ${}^\lambda\H_t(m,n;s,k)$ can be obtained simply by taking the transpose of an integer ${}^\lambda\H_t(n,m;k,s)$.
\end{proof}

The integer   ${}^6 \H_{20}(18, 15; 10,12 )$  shown in Figure \ref{big1} (on the left) has been obtained repeating
$\lambda_1=3$ times each of the blocks of Example \ref{18x10}.
In the same figure (on the right) we also give an integer
${}^{8}\H_5(16,20; 10,8)$, obtained repeating
$\lambda_1=2$ times each of the blocks of Example \ref{16x10}.

\begin{figure}
\rotatebox{90}{
\begin{footnotesize}
$\begin{array}{|c|c|c|c|c| c|c|c|c|c| c|c|c|c|c|}\hline
   \omb   1 & \omb -1 & \omb    &  -13 & -17 &     &    9 &  21 &     &   25 & -29 &     &  -33 &  37 &    \\ \hline
    \omb    &  \omb  2 &\omb  -2 &     &  -14 & -18 &     &   10 &  22 &     &   26 & -30 &     &  -34 &  38 \\ \hline
 \omb    -3 &  \omb   & \omb   3 & -19 &     &  -15 &  23 &     &   11 & -31 &     &   27 &  39 &     &  -35 \\ \hline
 \omb    -5 &  \omb 5 & \omb    &   17 &  13 &     &   -9 & -21 &     &  -25 &  29 &     &   33 & -37 &    \\ \hline
     \omb   &  \omb -6 & \omb  6 &     &   18 &  14 &     &  -10 & -22 &     &  -26 &  30 &     &   34 & -38 \\ \hline
 \omb     7 & \omb    & \omb  -7 &  15 &     &   19 & -23 &     &  -11 &  31 &     &  -27 & -39 &     &   35 \\ \hline
      1 &  -1 &     &  -13 & -17 &     &    9 &  21 &     &   25 & -29 &     &  -33 &  37 &    \\ \hline
        &    2 &  -2 &     &  -14 & -18 &     &   10 &  22 &     &   26 & -30 &     &  -34 &  38 \\ \hline
     -3 &     &    3 & -19 &     &  -15 &  23 &     &   11 & -31 &     &   27 &  39 &     &  -35 \\ \hline
     -5 &   5 &     &   17 &  13 &     &   -9 & -21 &     &  -25 &  29 &     &   33 & -37 &    \\ \hline
        &   -6 &   6 &     &   18 &  14 &     &  -10 & -22 &     &  -26 &  30 &     &   34 & -38 \\ \hline
      7 &     &   -7 &  15 &     &   19 & -23 &     &  -11 &  31 &     &  -27 & -39 &     &   35 \\ \hline
      1 &  -1 &     &  -13 & -17 &     &    9 &  21 &     &   25 & -29 &     &  -33 &  37 &    \\ \hline
        &    2 &  -2 &     &  -14 & -18 &     &   10 &  22 &     &   26 & -30 &     &  -34 &  38 \\ \hline
     -3 &     &    3 & -19 &     &  -15 &  23 &     &   11 & -31 &     &   27 &  39 &     &  -35 \\ \hline
     -5 &   5 &     &   17 &  13 &     &   -9 & -21 &     &  -25 &  29 &     &   33 & -37 &    \\ \hline
        &   -6 &   6 &     &   18 &  14 &     &  -10 & -22 &     &  -26 &  30 &     &   34 & -38 \\ \hline
      7 &     &   -7 &  15 &     &   19 & -23 &     &  -11 &  31 &     &  -27 & -39 &     &   35 \\ \hline
\end{array}$
\end{footnotesize}}
\quad
\rotatebox{90}{
\begin{footnotesize}
$\begin{array}{|c|c|c|c|c| c|c|c|c|c| c|c|c|c|c| c|c|c|c|c|}\hline
     1 &  -1 &    &    &   2 &  -2 &    &    &   3 &  -3 &    &    &   4 &  -4 &    &    &   5 &  -5 &    &    \\ \hline
       &   6 &  -6 &    &    &   7 &  -7 &    &    &   8 &  -8 &    &    &  10 & -10 &    &    &  11 & -11 &    \\ \hline
       &    &  12 & -12 &    &    &  13 & -13 &    &    &  14 & -14 &    &    &  15 & -15 &    &    &  16 & -16 \\ \hline
    -17 &    &    &  17 & -19 &    &    &  19 & -20 &    &    &  20 & -21 &    &    &  21 & -22 &    &    &  22 \\ \hline
     -1 &   1 &    &    &  -2 &   2 &    &    &  -3 &   3 &    &    &  -4 &   4 &    &    &  -5 &   5 &    &    \\ \hline
       &  -6 &   6 &    &    &  -7 &   7 &    &    &  -8 &   8 &    &    & -10 &  10 &    &    & -11 &  11 &    \\ \hline
       &    & -12 &  12 &    &    & -13 &  13 &    &    & -14 &  14 &    &    & -15 &  15 &    &    & -16 &  16 \\ \hline
     17 &    &    & -17 &  19 &    &    & -19 &  20 &    &    & -20 &  21 &    &    & -21 &  22 &    &    & -22 \\ \hline
      1 &  -1 &    &    &   2 &  -2 &    &    &   3 &  -3 &    &    &   4 &  -4 &    &    &   5 &  -5 &    &    \\ \hline
       &   6 &  -6 &    &    &   7 &  -7 &    &    &   8 &  -8 &    &    &  10 & -10 &    &    &  11 & -11 &    \\ \hline
       &    &  12 & -12 &    &    &  13 & -13 &    &    &  14 & -14 &    &    &  15 & -15 &    &    &  16 & -16 \\ \hline
    -17 &    &    &  17 & -19 &    &    &  19 & -20 &    &    &  20 & -21 &    &    &  21 & -22 &    &    &  22 \\ \hline
     -1 &   1 &    &    &  -2 &   2 &    &    &  -3 &   3 &    &    &  -4 &   4 &    &    &  -5 &   5 &    &    \\ \hline
       &  -6 &   6 &    &    &  -7 &   7 &    &    &  -8 &   8 &    &    & -10 &  10 &    &    & -11 &  11 &    \\ \hline
       &    & -12 &  12 &    &    & -13 &  13 &    &    & -14 &  14 &    &    & -15 &  15 &    &    & -16 &  16 \\ \hline
     17 &    &    & -17 &  19 &    &    & -19 &  20 &    &    & -20 &  21 &    &    & -21 &  22 &    &    & -22 \\ \hline
\end{array}$
\end{footnotesize}}
\caption{An integer  ${}^6 \H_{20}(18, 15; 10,12 )$  (on the left) and an  integer ${}^{8}\H_5(16,20;$ $10,8)$ (on the right).}
\label{big1}
\end{figure}

\subsection{The subcase $k\equiv 2\pmod 4$}

Here we only solve the case $m$ even, which implies that also $n$ is even.

\begin{prop}\label{sk2}
Suppose $6\leq s \leq n$, $6\leq k \leq m$, $ms=nk$ and $s,k\equiv 2 \pmod 4$.
Let $\lambda$ be a divisor of $2ms$ and let $t$ be a divisor of  $\frac{2ms}{\lambda}$.
If $m$ is even, there exists a shiftable integer ${}^\lambda\H_t(m,n;s,k)$.
\end{prop}

\begin{proof}
Without loss of generality, we may assume $m\geq n$ (and so $s\leq k$).
If $\lambda$ divides $ms$,  let $(\B_1,\B_2)$ be the nice pair of sequences constructed in  Proposition \ref{nice}.
Take $\B_1^*=\lambda_1 \ast \B_1$ and $\B_2^*=\lambda_1 \ast \B_2$.
So, $\B_1^*$ and $\B_2^*$  have length $\frac{m}{2}$ and $\mu(\B_1^*)=\mu(\B_2^*)=\lambda$.
If $\lambda$ does not divide $ms$, let $(\B_1^*,\B_2^*)$ be the nice pair of sequences constructed in 
Proposition \ref{nice 2ms}.
In both cases, write $\B_1^*=(B_1,\ldots,B_{\frac{m}{2}})$ and $\B_2^*=(B'_1,\ldots,B'_{\frac{m}{2}})$,
where $\B_1^*$ satisfies \eqref{blocchi}, $\B_2^*$ satisfies
\eqref{blocchiOLD} and
$$\supp(\B_1^*)=\supp(\B_2^*)=\left[ 1, \left\lfloor \frac{t\ell}{2} \right\rfloor\right]
\setminus \left\{j\ell : j \in [1, \lfloor t/2 \rfloor ]\right\}\quad \textrm{ with } \ell=\frac{2ms}{\lambda t}+1.$$
Set
$$\wt \B_1= \left(B_{\frac{n}{2}+1},\ldots, B_{\frac{m}{2}} \right) \equad \wt \B_2=\left(B_1',\ldots,B'_{\frac{n}{2}}
\right).$$
Since, by construction, $\E(B_i)=\E(B'_i)$ for all $i\in [1,\frac{m}{2}]$, it follows that
$\E(\wt \B_2\con \wt \B_1)=\E(\B_1^*)=\E(\B_2^*)$ and
$\supp(\wt \B_2\con \wt \B_1)=\left[ 1,\left\lfloor \frac{t\ell}{2} \right\rfloor\right]
\setminus \{j\ell : j \in [1, \lfloor t/2 \rfloor ]$.
Furthermore, if $\lambda$ divides $ms$ then $\mu(\wt \B_2\con \wt \B_1)=\lambda$;
the same holds if $\lambda$ does not divide $ms$, and $\ell$ is odd or $t$ is even;
if $\lambda$ does not divide $ms$,  $\ell$ is even and  $t$ is odd, 
then every element of $\Phi \setminus \left\{\frac{t\ell}{2}\right\}$ appears in 
$\E(\wt \B_2\con \wt \B_1)$, up to sign, exactly $\lambda$ times, while the integer 
$\frac{t\ell}{2}$  appears, up to sign, $\frac{\lambda}{2}$ times.

Using the blocks of the sequence $\wt \B_2$, we first construct a square shiftable p.f. array $A_1$ of size $n$
such that each row and each column contains $s$ filled cells and such that the elements in every row and column sum to
zero. Hence, take an empty array  $A_1$ of size $n\times n$ and arrange the $\frac{n}{2}$ blocks $B'_r=(b_{i,j}^{(r)})$
of $\wt \B_2$ in
such a way that the element $b_{1,1}^{(r)}$ fills the cell $(2r-1,2r-1)$ of $A_1$.
This process makes $A_1$ a p.f. array with $s$ filled cells in each row and in each column.
Since the rows of the blocks $B_r'$ sum to zero, also the rows of $A_1$ sum to zero.
Looking at the columns, the $s$ elements of a column of $A_1$ are
$$b_{1,s}^{(r)}, b_{2,s}^{(r)},\; b_{1,s-2}^{(r+1)}, b_{2,s-2}^{(r+1)},\;
b_{1,s-4}^{(r+2)}, b_{2,s-4}^{(r+2)}, \;\ldots,\;  b_{1,2}^{(r+s/2)}, b_{2,2}^{(r+s/2)}$$
or
$$b_{1,s-1}^{(r)}, b_{2,s-1}^{(r)},\; b_{1,s-3}^{(r+1)}, b_{2,s-3}^{(r+1)},\;
b_{1,s-5}^{(r+2)}, b_{2,s-5}^{(r+2)}, \;\ldots,\;  b_{1,1}^{(r+s/2)}, b_{2,1}^{(r+s/2)},$$
where the exponents $r,\ldots,r+s/2$ must be read modulo $\frac{n}{2}$.
Since $\wt \B_2$ satisfies condition \eqref{blocchiOLD}, the sum of these elements is
$$\sum_{j=1}^{s/2}\sigma_{2j}=0 \quad \textrm{ or }\quad  \sum_{j=1}^{s/2}\sigma_{2j-1}=0, \quad \textrm{ respectively.}$$
By construction, $\E(A_1)=\E(\wt \B_2)$.

Now, if $m=n$, then $A_1$ is actually a shiftable integer ${}^\lambda\H_t(m,n;k,s)$.
Suppose that $m>n$. If we arrange the blocks of the sequence $\wt \B_1$
mimicking what we did for the construction of an integer ${}^1\H_1(m-n, n; s, k-s)$  in
the proof of Proposition \ref{prop s2},
we obtain a shiftable p.f. array $A_2$ of size $(m-n)\times n$ such that
$\E(A_2)=\E(\wt \B_1)$, rows and columns sum to zero, each row contains $s$ filled cells and each column
contains $k-s$ filled cells.
Let $A$ be the p.f. array of size $m\times n$ obtained taking
$$A=\begin{array}{|c|}\hline
   A_1\\\hline
   A_2 \\\hline
   \end{array}.$$
Each row of  $A$ contains $s$ filled cells and each of its columns contains $s+(k-s)=k$ filled cells.
By the previous properties of $\wt \B_2\con \wt \B_1$, it follows that
$A$ is a shiftable integer ${}^\lambda\H_t(m,n;s,k)$.
\end{proof}

An integer  ${}^{28}\H_4(16,16;14,14)$ is shown in Figure \ref{big2} (on the left),
choosing $\lambda_1=2$ and $\lambda_2=14$.
In the same figure (on the right) we also give an integer
${}^{10}\H_3(20,12;6,10)$, where $\lambda_1=5$ and $\lambda_2=2$.

\begin{figure}
\rotatebox{90}{
\begin{footnotesize}
$\begin{array}{|c|c|c|c|c|c|c|c|c|c|c|c|c|c|c|c|}\hline
   1 &  2 & -1 &  1 & -1 & -2 &  1 & -1 &  2 & -2 &  1 & -1 &  2 & -2 & &  \\ \hline
   -2 & -1 &  2 & -2 &  1 &  2 & -1 &  1 & -2 &  2 & -1 &  1 & -2 &  2 & &  \\ \hline
   && 3 &  4 & -3 &  3 & -3 & -4 &  3 & -3 &  4 & -4 &  3 & -3 &  4 & -4 \\ \hline
   &&-4 & -3 &  4 & -4 &  3 &  4 & -3 &  3 & -4 &  4 & -3 &  3 & -4 &  4 \\ \hline
    7 & -7 & && 6 &  7 & -6 &  6 & -6 & -7 &  6 & -6 &  7 & -7 &  6 & -6 \\ \hline
   -7 &  7 & &&-7 & -6 &  7 & -7 &  6 &  7 & -6 &  6 & -7 &  7 & -6 &  6 \\ \hline
    8 & -8 &  9 & -9 & && 8 &  9 & -8 &  8 & -8 & -9 &  8 & -8 &  9 & -9 \\ \hline
   -8 &  8 & -9 &  9 & &&-9 & -8 &  9 & -9 &  8 &  9 & -8 &  8 & -9 &  9 \\ \hline
    2 & -2 &  1 & -1 &  2 & -2 & && 1 &  2 & -1 &  1 & -1 & -2 &  1 & -1 \\ \hline
   -2 &  2 & -1 &  1 & -2 &  2 & &&-2 & -1 &  2 & -2 &  1 &  2 & -1 &  1 \\ \hline
    3 & -3 &  4 & -4 &  3 & -3 &  4 & -4 & && 3 &  4 & -3 &  3 & -3 & -4 \\ \hline
   -3 &  3 & -4 &  4 & -3 &  3 & -4 &  4 & &&-4 & -3 &  4 & -4 &  3 &  4 \\ \hline
   -6 & -7 &  6 & -6 &  7 & -7 &  6 & -6 &  7 & -7 & && 6 &  7 & -6 &  6 \\ \hline
    6 &  7 & -6 &  6 & -7 &  7 & -6 &  6 & -7 &  7 & &&-7 & -6 &  7 & -7 \\ \hline
   -8 &  8 & -8 & -9 &  8 & -8 &  9 & -9 &  8 & -8 &  9 & -9 & && 8 &  9 \\ \hline
    9 & -9 &  8 &  9 & -8 &  8 & -9 &  9 & -8 &  8 & -9 &  9 & &&-9 & -8 \\ \hline
\end{array}$
\end{footnotesize}}
\quad
\rotatebox{90}{
\begin{footnotesize}
$\begin{array}{|c|c|c|c|c| c|c|c|c|c| c|c|}\hline
    1&  -1&  -4&  -5&   3&   6&  & & & & &   \\ \hline
    -2&   2&   5&   4&  -3&  -6&  & & & & &   \\ \hline
    & &  7&  -7& -11& -12&  10&  13&  & & &   \\ \hline
    & & -8&   8&  12&  11& -10& -13&  & & &   \\ \hline
    & & & &  1&  -1&  -4&  -5&   3&   6&  &   \\ \hline
    & & & & -2&   2&   5&   4&  -3&  -6&  &   \\ \hline
    & & & & & &  7&  -7& -11& -12&  10&  13 \\ \hline
    & & & & & & -8&   8&  12&  11& -10& -13 \\ \hline
     3&   6&  & & & & & &  1&  -1&  -4&  -5 \\ \hline
    -3&  -6&  & & & & & & -2&   2&   5&   4 \\ \hline
   -11& -12&  10&  13&  & & & & & &  7&  -7 \\ \hline
    12&  11& -10& -13&  & & & & & & -8&   8 \\ \hline
     1&  -1&  & & -4&  -5&  & &  3&   6&  &  \\ \hline
    &  7&  -7&  & &-11& -12&  & & 10&  13&    \\ \hline
    & &  1&  -1&  & & -4&  -5&  & &  3&   6 \\ \hline
    -7&  & &  7& -12&  & &-11&  13&  & & 10 \\ \hline
    -2&   2&  & &  5&   4&  & & -3&  -6&  &   \\ \hline
    & -8&   8&  & & 12&  11&  & &-10& -13&    \\ \hline
    & & -2&   2&  & &  5&   4&  & & -3&  -6 \\ \hline
     8&  & & -8&  11&  & & 12& -13&  & &-10 \\ \hline
\end{array}$
\end{footnotesize}}
\caption{An integer  ${}^{28}\H_4(16,16;14,14)$  (on the left) and an  integer ${}^{10}\H_3(20,12;6,10)$ (on the right).}
\label{big2}
\end{figure}

\section{Conclusions}

Thanks to the constructions of Sections \ref{s0k0} and \ref{s2}, we can prove Theorem \ref{mainH}.
In fact, case (1) follows from Proposition \ref{prop:k4}; cases (2) and (3) follow from Proposition \ref{prop s2}; 
case (4) follows from Proposition \ref{sk2}.
Unfortunately, we are not able to solve the existence of an integer ${}^{\lambda}\H_t(m,n;s,k)$
when $s,k\equiv 2\pmod 4$ and $m,n$ are odd.
However, we can prove the existence of an $\SMA(m,n;s,k)$  for this choice of $m,n,s,k$.

\begin{proof}[Proof of {\rm Theorem \ref{main}}]
If $s,k\equiv 0\pmod 4$, the integer ${}^2\H_1(m,n;s,k)$ we construct in Lemma 	\ref{lambda2} is actually a (shiftable) 
$\SMA(m,n;s,k)$.
Similarly, if $s\equiv 2 \pmod 4$ and $m$ is even, the integer ${}^2\H_1(m,n;s,k)$ constructed
in  Propositions \ref{prop s2} and \ref{sk2} are (shiftable) signed magic arrays.
So, we are left to consider the case $s,k\equiv 2\pmod 4$ with $m,n$ odd.

Without loss of generality, we may assume $m\geq n$ (and so $s\leq k$).
Let $A_1$ be an $\SMA(n,n;s,s)$, whose existence is assured by Theorem \ref{square}.
Clearly if $m=n$ we have nothing to prove. So, suppose $m>n$.
Since $m-n\geq 2$ is even and $k-s\equiv 0 \pmod 4$ with $k-s\geq 4$, by
Proposition \ref{prop s2} there exists a shiftable $\SMA(m-n, n;s,k-s)$, say $A_2$.
Let $A$ be the p.f. array of size $m\times n$ obtained taking
$$A=\begin{array}{|c|}\hline
   A_1\\\hline
   A_2\pm ns/2 \\\hline
   \end{array}.$$
Each row of  $A$ contains $s$ filled cells and each of its columns contains $s+(k-s)=k$ filled cells.
Also, note that $\E(A_1)=\{\pm 1, \pm 2, \ldots, \pm ns/2\}$ and
$\E(A_2\pm sn/2)=\{\pm (1+ns/2), \pm (2+ns/2), \ldots, \pm ms/2\}$.
Since  $\E(A)=\E(A_1)\cup \E(A_2)=\{\pm 1, \pm 2, \ldots, \pm ms/2\}$,  $A$ is an $\SMA(m,n;s,k)$.
\end{proof}

We can now prove the existence of magic rectangles.

\begin{proof}[Proof of {\rm Theorem \ref{mainR}}]
Let $A$ be a shiftable $\SMA(m,n;s,k)$, whose existence was proved in Theorem \ref{main},
and let $A^*$ be the p.f. array obtained by replacing every negative entry $x$ of $A$ with $x+\frac{ms}{2}$
and by replacing every positive entry $y$  with $y+\frac{ms}{2}-1$.
Since $\E(A)=\left\{-1, -2, \ldots, -\frac{ms}{2}\right\}\cup \left\{1,2,\ldots,\frac{ms}{2}\right\}$,
we obtain $\E(A^*)=\left\{0,1,\ldots,\frac{ms}{2}-1 \right\} \cup 
\left\{\frac{ms}{2},\frac{ms}{2}+1, \ldots,\right.$ $\left. ms-1 \right\}$.
This means that every element of $[0,ms-1]$ appears just once in $A^*$. Obviously, every row of $A^*$ contains exactly $s$ filled cells
and every column of $A^*$ contains exactly $k$ filled cells.
Now, since $A$ is shiftable, every row of $A$ contains $\frac{s}{2}$ negative entries and $\frac{s}{2}$ positive entries.
So, the elements of every row of $A^*$ sum to $\frac{s}{2}\left(\frac{ms}{2}+\frac{ms}{2}-1 \right)=\frac{s(ms-1)}{2}$.
Analogously, the elements of every column of $A^*$ sum to $\frac{k(ms-1)}{2}$. 
We conclude that $A^*$ is an $\MR(m,n;s,k)$.
\end{proof}

\begin{ex}
Take the following shiftable $\SMA(5,10;8,4)$, whose construction is given in Lemma \ref{lambda2}:

\begin{footnotesize}
$$A=\begin{array}{|c|c|c|c|c|c|c|c|c|c|}\hline
1  & -2 &    & -7 &  8 & 11  & -12 &     & -17 & 18\\\hline
20 &  3 & -4 &    & -9 & 10  & 13  & -14 &     & -19 \\\hline
-1 & 2  &  5 & -6 &    & -11 & 12  & 15  & -16 &    \\\hline
   & -3 &  4 &  7 & -8 &     & -13 & 14  &  17 & -18\\\hline
-20&    & -5 &  6 &  9 & -10 &     & -15 & 16  & 19 \\\hline
    \end{array}.$$
\end{footnotesize}

\noindent Proceeding as described in the proof of Theorem \ref{mainR}, we obtain the following $\MR(5,10;8,4)$:

\begin{footnotesize}
$$A^*=\begin{array}{|c|c|c|c|c|c|c|c|c|c|}\hline
20  & 18 &    & 13 &  27 & 30  & 8 &     & 3 & 37 \\\hline
39 &  22 & 16 &    & 11 & 29  & 32  & 6 &     & 1 \\\hline
19 & 21  &  24 & 14 &    & 9 & 31  & 34  & 4 &    \\\hline
   & 17 &  23 &  26 & 12 &     & 7 & 33  &  36 & 2   \\\hline
0 &    & 15 &  25 &  28 & 10 &     & 5 & 35  & 38 \\\hline
    \end{array}.$$
\end{footnotesize}
\end{ex}


\begin{thebibliography}{20}

\bibitem{A} D.S. Archdeacon,
Heffter arrays and biembedding graphs on surfaces,
\textit{Electron. J. Combin.} \textbf{22} (2015), \#P1.74.

\bibitem{ABD} D.S. Archdeacon, T. Boothby, J.H. Dinitz,
Tight Heffter arrays exist for all possible values,
\textit{J. Combin. Des.} \textbf{25} (2017), 5--35.

\bibitem{ADDY} D.S. Archdeacon, J.H. Dinitz, D.M. Donovan, E.S. Yaz{\i}c{\i},
Square integer Heffter arrays with empty cells,
\textit{Des. Codes Cryptogr.} \textbf{77} (2015), 409--426.

\bibitem{BCDY} K. Burrage, D.M. Donovan, N.J. Cavenagh,  E.\c{S}. Yaz{\i}c{\i},
Globally simple Heffter arrays $H(n;k)$ when $k\equiv 0,3 \pmod{4}$,
\textit{Discrete Math.} \textbf{343} (2020), \#111787.

\bibitem{CDDY} N.J. Cavenagh, J.H. Dinitz, D.M. Donovan,  E.\c{S}. Yaz{\i}c{\i},
The existence of square non-integer Heffter arrays,
\textit{Ars Math. Contemp.} \textbf{17} (2019), 369--395.

\bibitem{CDYBiem} N.J. Cavenagh, D.M. Donovan,  E.\c{S}. Yaz{\i}c{\i},
Biembeddings of cycle systems using integer Heffter arrays,
to appear in \textit{J. Combin. Des}, \texttt{https://doi.org/10.1002/jcd.21753}.

\bibitem{CMPPHeffter} S. Costa, F. Morini, A. Pasotti, M.A. Pellegrini,
Globally simple Heffter arrays and orthogonal cyclic cycle decompositions,
\textit{Austral. J. Combin.} \textbf{72} (2018), 549--593.

\bibitem{RelH} S. Costa, F. Morini, A. Pasotti, M.A. Pellegrini,
A generalization of Heffter arrays,
\textit{J. Combin. Des.} \textbf{28} (2020), 171--206.

\bibitem{SA} S. Costa, A. Pasotti,
On $\lambda$-fold relative Heffter arrays and biembedding multigraphs on surfaces,
preprint available at \texttt{https://arxiv.org/abs/2010.10948}.

\bibitem{RelHBiem} S. Costa, A. Pasotti, M.A. Pellegrini,
Relative Heffter arrays and biembeddings,
\textit{Ars Math. Contemp.} \textbf{18} (2020), 241--271.
 
\bibitem{DM} J.H. Dinitz, A.R.W. Mattern,
Biembedding Steiner triple systems and $n$-cycle systems on orientable surfaces,
\textit{Austral. J. Combin.} \textbf{67} (2017), 327--344.

\bibitem{DW} J.H. Dinitz, I.M. Wanless,
The existence of square integer Heffter arrays,
\textit{Ars Math. Contemp.} \textbf{13} (2017), 81--93.

\bibitem{magic2} A. Khodkar, B. Ellis,
Signed magic rectangles with two filled cells in each column,
preprint available at \texttt{https://arxiv.org/abs/1901.05502}.


\bibitem{rect3} A. Khodkar, D. Leach, 
Magic rectangles with empty cells,
to appear in \textit{Util. Math.},
preprint available at  \texttt{https://arxiv.org/abs/1809.08605}.


\bibitem{rect1} A. Khodkar, D. Leach, 
Magic squares with empty cells,
to appear in \textit{Ars Combinatoria},
preprint available at \texttt{https://arxiv.org/abs/1804.11189}.

\bibitem{magic3} A. Khodkar, D. Leach, B. Ellis,
Signed magic rectangles with three filled cells in each column, 
\textit{Bull. Inst. Combin. Appl.} \textbf{90} (2020), 87--106. 

\bibitem{magicDM} A. Khodkar, C. Schulz, N. Wagner,
Existence of some signed magic arrays,
\textit{Discrete Math.} \textbf{340} (2017), 906--926.

\bibitem{MP} F. Morini, M.A. Pellegrini,
On the existence of integer relative Heffter arrays,
\textit{Discrete Math.} \textbf{343} (2020), \#112088.


\end{thebibliography}
\end{document}